\documentclass[12pt]{article}
\usepackage[normalem]{ulem}
\usepackage{hyperref}
\usepackage{amsmath,dsfont}
\usepackage{amsthm}
\usepackage{enumerate}
\usepackage{amssymb}
\usepackage[latin1]{inputenc}
\usepackage{color}
\usepackage{graphicx}
\begin{document}
\newcommand{\R}{{\bf R}}
\newcommand{\K}{{\bf K}}
\newcommand{\C}{{\bf C}}
\newcommand{\eqla}{\overset{({\cal L})}{=}}
\newcommand{\cola}{\xrightarrow{({\cal L})}}
\newcommand{\sequ}[2]{(#1_{#2})_{#2\geq 1}}
\newcommand{\rsequ}[1]{(#1_n)_{n=1,2,\dots}}
\newcommand{\seq}[1]{(#1(n))_{n=0,1,\dots}}
\newcommand{\rseq}[1]{(#1_n)_{n=0,1,\dots}}
\newcommand{\css}[1]{{\color{red}#1}}
\newcommand{\cs}[1]{{\color{blue}#1}}
\newcommand{\csl}[1]{{\color{magenta}#1}}
\newcommand{\rf}[1]{(\ref{#1})}
\newtheorem{thm}{Theorem}[section]
\newtheorem{example}{Example}
\newtheorem{defn}[thm]{Definition}
\newtheorem{cor}[thm]{Corollary}
\newcommand{\rtref}[1]{{\rm \ref{#1}}}
\newcommand{\rmref}[1]{{\rm (\ref{#1})}}
\newcommand{\rmcite}[1]{{\rm \cite{#1}}}
\newcommand{\rmciter}[2]{{\rm \cite[#1]{#2}}}
\newtheorem{lem}[thm]{Lemma}
\newtheorem{prop}[thm]{Proposition}
\newtheorem{rem}[thm]{Remark}

	\title {Explicit confidence bands and intervals for distribution functions and their derivatives via random Weierstrass-type operators}
	\author{José A. Adell,\footnote{e-mail: adell@unizar.es, Orcid ID: 0000-0001-8331-5160 } \ \ J. T. Alcalá\footnote{e-mail: jtalcala@unizar.es, Orcid ID: 0000-0001-7549-8825} \\ and  \ \  C. Sangüesa\footnote{e-mail: csangues@unizar.es, Orcid ID: 0000-0002-7099-7665}
		\\ \small Department of Statistical Methods and IUMA, \\ \small University of Zaragoza, Zaragoza, 50009,  SPAIN}

	\date{}

	\begin{titlepage}
		\setcounter{page}{1} \maketitle
		
		\bigskip \bigskip
		\begin{abstract}
			Classical kernel estimators of second order are interpreted in terms of random Weierstrass-type operators, particularly random Steklov operators. This leads us to obtain explicit nonasymptotic confidence bands and intervals for distribution functions $F$ and their derivatives $F^{(k)}$. Under the only assumption that $F^{(k)}$ is uniformly continuous, confidence bands for $F^{(k)}$ are established by using the Dvoretzky-Kiefer-Wolfowitz inequality. To give confidence intervals, we allow $F^{(k)}$ to have isolated discontinuities of the first kind, so that we really estimate the midpoint function $(F^{(k)})_{\star}(x)$. The proofs are based either on concentration inequalities for subordinated stochastic processes or accurate estimates of the MSE of the corresponding estimators. The length of the confidence bands and intervals depends on the degree of smoothness of $F^{(k)}$ measured in terms of the second modulus of continuity. Both lengths are of order $n^{-1 / 2}$ if $F$ is locally a polynomial of degree $k+1$ at most.
			
		\end{abstract}

		\bigskip \bigskip
		2000 Mathematics subject classification: Primary: 62G05, 60E05; secondary: 41A25, 41A36
		
		\bigskip \bigskip
		{\it Key words and phrases}: random Weierstrass-type operator, random Steklov operator, distribution function estimator, density estimator, confidence band, confidence interval, Dvoretzky-Kiefer-Wolfowitz inequality, second modulus of smoothness
		\bigskip \bigskip

		\bigskip \noindent

	\end{titlepage}

	\maketitle
	

\section{Introduction}
Since the pioneering papers by Rosenblat \cite{rorema} and Parzen \cite{paones}, kernel estimation is a powerful technique successfully applied to different branches of statistics, particularly, to the estimation of distribution functions $F$ and their derivatives (see, for instance, Silverman \cite{sidens}, Wand and Jones \cite{wajoke}, Tsybakov \cite{tsyba}, and Scott \cite{scott}, among others). Most of the results found in the literature have an asymptotic character, with no specified rates of convergence.

The aim of this paper is to give explicit nonasymptotic confidence bands and intervals for distribution functions and their derivatives. This is achieved through an interpretation of second-order kernel estimators in terms of random Weierstrass-type positive linear operators.

The main features of this paper are the following. In Section \ref{seweie}, we consider Weierstrass-type operators, mainly Steklov operators of order $m$ (see definition (5)). For these last operators, we give closed form expressions for their derivatives which are easy to handle, as well as local and uniform approximation properties in terms of the second modulus of smoothness of the function under consideration. Local approximation properties hold even if the function has isolated discontinuities of the first kind.

In Section \ref{serand}, we introduce random Weierstrass-type operators and show that classical kernel estimators of second order can be interpreted by means of such operators and their derivatives. For the sake of concreteness, we mainly use in this paper random Steklov operators of order $m$. It turns out that the derivatives of these last random operators are easy to compute. For reasons that will be clear further on, we use the $k$th derivative of the random Steklov operator of order $k+1$ to estimate $F^{(k)}, k=0,1,2, \ldots$ All these random estimators have continuous paths.

In Section \ref{seconf}, we obtain explicit confidence bands for $F^{(k)}$ under the only assumption that $F^{(k)}$ is uniformly continuous. The length of the corresponding confidence band strongly depends on the degree of smoothness of $F^{(k)}$, measured in terms of its second modulus of continuity. In the special case in which $F$ is a polynomial of degree $k+1$ at most on a certain closed interval, the length of the confidence band for $F^{(k)}$ restricted to a bit smaller closed interval is of order $n^{-1 / 2}$. A unified proof based on the Dvoretzky-Kiefer-Wolfowitz inequality is given.

Explicit confidence intervals for distribution functions and their derivatives are collected in Section \ref{seinte}. We emphasize that we allow $F^{(k)}, \ k=0,1,2\dots$, to have isolated discontinuities of the first kind, so that we really estimate the midpoint function $(F^{(k)})_{\star}(x)$ defined in (\ref{D1}) below.

To give confidence intervals for $F_{\star}(x)$, we use random Steklov operators of first order acting on the empirical process. The mean square error (MSE) of such estimators is computed in closed form. In general, the length of the confidence intervals is of order $n^{-1 / 2}$. However, if $F$ is continuous and concentrated on some finite interval $[a, b]$, the length of the confidence intervals at $a$ and $b$ is as short as we wish. An application to discrete random variables is also provided.

Suppose that the probability density $\rho=F^{(1)}$ has isolated discontinuities of the first kind. We use the first derivative of the random Steklov operator of second order as the estimator of $\rho_{\star}(x)$ and exactly compute its MSE. The lengths of the corresponding confidence intervals depend on the local second modulus of smoothness of the auxiliary function $\rho_{x}$ defined in (\ref{D2}). The following particular cases deserve to be mentioned.  If $\rho_{\star}(x)=0$, the length of the confidence interval is asymptotically shorter, whereas if $\rho$ is locally affine, the length is of order $n^{-1 / 2}$. Finally, if $x$ is close to a discontinuity point of $\rho$, we need to take a large sample size to construct the corresponding confidence interval. This is illustrated in Section \ref{sesimu} by considering the Pareto density.

For $k \geq 1$, we use the $k+1$ derivative of the random Steklov operator of order $k+2$ as the estimator of the midpoint function $\left(\rho^{(k)}\right)_{\star}(x)$ and give accurate estimates of its MSE. Again, the length of the confidence intervals depends upon the local second modulus of smoothness of the function $\left(\rho^{(k)}\right)_{x}$. Such a length is of order $n^{-1/ 2}$ if $\rho$ is locally a polynomial of degree $k+1$ at most.

All the proofs concerning confidence intervals, moved to the Appendix, are based on two main tools: either a concentration inequality involving subordinated stochastic processes stated in Lemma \ref{lea1}, or an accurate estimate of the MSE of the estimator under consideration together with Chebyshev's inequality, according to the magnitude of the MSE. Finally, let us say that, in a strict sense, asymptotic and nonasymptotic confidence bands and intervals are not comparable. This notwithstanding, a comparative discussion with other results in the literature is carried out throughout Sections \ref{seconf} and \ref{seinte}.
	
	\section{Weierstrass-type operators} \label{seweie}
	
Let $\mathbb{N}$ be the set of positive integers and $\mathbb{N}_0=\mathbb{N}\cup \{0\}$. Unless otherwise specified, we assume from now on that $x\in \mathbb{R}$, $h>0$, and $m,n\in \mathbb{N}$.  Let $S\subseteq \mathbb{R}$ be either the empty set or a countable ordered set $S=\{x_j,\ j\in \mathbb{Z}\}$ such that
\begin{equation}\inf_{j\in \mathbb{Z} }(x_{j+1}-x_j)=s>0.\label{dfsets}\end{equation} Denote by $ {\cal C}_{S}(\mathbb{R})$ the set of bounded
functions, which are continuous on $\mathbb{R}\setminus S$ and have right and left limits at each point of $S$.   The set of bounded continuous functions is simply denoted by ${\cal C}(\mathbb{R})={\cal C}_{\emptyset}(\mathbb{R})$.
If $f \in {\cal C}_{S}(\mathbb{R})$, we denote
\begin{equation}
	f_{\star}(y)=\frac{1}{2}\left(f\left(y+\right)+f\left(y-\right)\right), \quad y \in \mathbb{R}, \label{D1}
\end{equation}
where $f\left(y+\right)$ and $f\left(y-\right)$ are the right and the left limits of $f$ at $y$, respectively. If $k \in \mathbb{N}$, we denote by ${\cal C}_{S}^{k}(\mathbb{R})$ the set of $k-1$ times differentiable functions such that $f^{(k-1)}$ has a Radon-Nikodym derivative $f^{(k)} \in {\cal C}_{S}(\mathbb{R})$. We also set ${\cal C}^{k}(\mathbb{R})={\cal C}_{\emptyset}^{k}(\mathbb{R})$.  Finally, let $(\widetilde{\Omega},\widetilde{\cal F}, \widetilde{P})$ be a probability space. Expectations and random variables defined on this space are denoted by $\widetilde{\mathbb{E}}$ and $\widetilde{X}$, respectively.

Associated to  a random variable $\widetilde{W}$, we define the Weierstrass-type operator
\begin{equation}
	L_h(f;x)=\mathbb{\widetilde{E}}f(x+h\widetilde{W}),\quad f\in {\cal C}_{S}(\mathbb{R}). \label{dfweie}
\end{equation}
If $\widetilde{W}$ has the standard normal distribution, then $L_h$ is called the Weierstrass operator.  For computational reasons, we consider in this paper the following particular case of (\ref{dfweie}).  Let $\sequ{\widetilde{V}}{j}$ be a sequence of independent copies of a random variable $\widetilde{V}$ uniformly distributed on $[-1,1]$ and set
\begin{equation}
	\widetilde{S}_m=\widetilde{V}_1+\dots+\widetilde{V}_m. \label{dfsuni}
\end{equation}
We define the Steklov operator of order $m$ as

\begin{equation}
	L_{m,h}(f;x)=\mathbb{\widetilde{E}}f(x+h\widetilde{S}_m),\quad f\in {\cal C}_{S}(\mathbb{R}). \label{dfstek}
\end{equation}
These and many other positive linear operators can be found in Altomare and Campiti \cite{alto}.

Recall that the $m$th symmetric differences of $f\in {\cal C}_{S}(\mathbb{R})$ at step $h$ are recursively defined as
\begin{equation}
	\Delta_h^0 f(x)=f(x), \ 	\Delta_h^1f(x)=f(x+h)-f(x-h),\ \Delta_h^mf(x)=\Delta_h^1\left(\Delta_h^{m-1}f\right)(x), \label{dfsydi}
\end{equation}
or, equivalently, by
\begin{equation}
	\Delta_h^mf(x)=\sum_{j=0}^m {m \choose j}(-1)^{m-j}f(x+(2j-m)h).\label{dfdire}
\end{equation}
Denote by $f_{(m)}$ an antiderivative of $f\in  {\cal C}_{S}(\mathbb{R})$ of order $m$, that is,
\[f_{(0)}(x)=f(x),\quad f_{(1)}(x)=\int_0^xf(u)du,\quad f_{(m)}(x)=\int_0^xf_{(m-1)}(u)du. \]
It can be seen by induction on $m$ that
\begin{equation}
		L_{m,h}(f;x)=\frac{1}{(2h)^m}\Delta_h^mf_{(m)}(x), \label{rededi}
\end{equation}
thus implying that
\begin{equation}
\Delta_h^mg(x)=(2h)^m	L_{m,h}(g^{(m)};x),\quad  g\in {\cal C}^m_{S}(\mathbb{R}), \label{restdi}
\end{equation}
as follows by setting $f=g^{(m)}$ in (\ref{rededi}).
The following result can be derived from Feller (cf. \cite[Th. 1.a]{feanin}).  We give here a short proof of it for the sake of completeness.
\begin{prop} Let $\widetilde{F}_m$ be the distribution function of $\widetilde{S}_m$. Then,
	\[\widetilde{F}_m^{(k+1)} (u)=\frac{1}{2^m(m-k-1)!}\sum_{j=0}^{m}{m \choose j}(-1)^j(u-(2j-m))_{+}^{m-k-1},\quad u\in\mathbb{R},\]
	where $k=-1,0,\dots,m-2$, and $y_+=\max(0,y),\ y\in \mathbb{R}$. \label{prdesu}
\end{prop}
\begin{proof} Fix $u \in \mathbb{R}$ and consider the function $f(\theta)=1_{(-\infty, u]}(\theta),\ \theta \in \mathbb{R}$. Observe that
	\[
	f_{(m)}(\theta)=\frac{(-1)^{m}}{m!}(u - \theta)_{+}^{m}, \quad \theta \in \mathbb{R} .
	\]
	Hence, we have from (\ref{dfsydi}), (\ref{dfdire}), and (\ref{restdi})
	\begin{align*}
		&\widetilde{F}_{m}(u)=P\left(\widetilde{S}_{m} \leq u\right)=L_{m, 1}(f ; 0)=\frac{1}{2^{m}} \Delta_{1}^{m} f_{(m)}(0)\\
		&=\frac{1}{2^{m}} \sum_{j=0}^{m}\binom{m}{j}(-1)^{m-j} f_{(m)}(2 j-m)=\frac{1}{2^{m} m!} \sum_{j=0}^{m}\binom{m}{j}(-1)^{j}\left(u -(2 j-m)\right)_{+}^{m}.
	\end{align*}
	Differentiating this formula, we obtain the result.
\end{proof}
The first interesting property of Steklov operators is that they take non-smooth into smooth functions, as shown in the following result.
\begin{prop}
	We have $L_{m,h}\left({\cal C}_{S}(\mathbb{R})\right)\subseteq {\cal C}^{m-1}(\mathbb{R})$. Specifically,
	\begin{equation}
		L_{m,h}^{(k)}(f;x)=\frac{1}{(2h)^k}\Delta_h^k	L_{m-k,h}(f;x), \quad k=0,1,\dots,m-1. \label{stekde}
	\end{equation}
	In addition, if  $f\in {\cal C}^{k}_{S}(\mathbb{R})$, then
	\begin{equation}
		L_{m,h}^{(k)}(f;x)=	L_{m,h}(f^{(k)};x), \quad k=0,1,\dots, m-1. \label{stekd2}
	\end{equation} \label{prstde}
	
\end{prop}
\begin{proof} By (\ref{dfsydi}) and (\ref{rededi}), we have
	\begin{align*}
		L_{m,h}^{(k)}(f;x)=\frac{1}{(2h)^m}\Delta_h^mf_{(m-k)}(x)=\frac{1}{(2h)^k}\Delta_h^k\frac{\Delta_h^{m-k}f_{(m-k)}(x)}{(2h)^{m-k}},
	\end{align*}
	thus showing (\ref{stekde}).  Again from (\ref{rededi}), we have
	\[L_{m,h}(f^{(k)};x)=	\frac{1}{(2h)^m}\Delta_h^mf_{(m-k)}(x)=	L_{m,h}^{(k)}(f;x).\]
	This shows (\ref{stekd2}) and completes the proof.\end{proof}
To give confidence intervals for the derivatives of distribution functions (see Section \ref{seinte}), we will use the following Weierstrass-type operators. For any $k \in \mathbb{N}_{0}$, let $\widetilde{T}_{k}$ be the symmetric random variable
\begin{equation}
	\widetilde{T}_{k}=2 \widetilde{R}_{k} - k, \quad P\left(\widetilde{R}_{k}=j\right)= \frac{\binom{k}{j}^{2}}{\binom{2 k}{k}}, \quad j=0,1, \dots, k .  \label{A}
\end{equation}
Suppose that $\widetilde{T}_{k}$ is independent of $\widetilde{S}_{m - k}, \ k<m$, as defined in (\ref{dfsuni}). We consider the Weiestrass-type operator
\begin{equation}D_{m, k, h}(f ; x)=\widetilde{\mathbb{E}} f\left(x+h\left(\widetilde{T}_{k}+\widetilde{S}_{m-k}\right)\right),\quad k \in \mathbb{N}_{0}, \quad k<m,\quad f \in {\cal C}_{S}(\mathbb{R}).\label{B}\end{equation}
Observe that
\begin{equation}
	D_{m, 0, h}=L_{m, h} . \label{C}
\end{equation}
The local and uniform rates of convergence of these operators can be measured in terms of the second modulus of the function under consideration. To be more precise, let $I$ be a subinterval of $\mathbb{R}$. The local second modulus of smoothness of $f \in {\cal C}_{S}(\mathbb{R})$ is defined as
\[
\omega_{2}(f, I ; \delta)=\sup\{|f(x-\epsilon)-2 f(x)+f(x+\epsilon)|:\ [x-\epsilon, x + \epsilon] \subseteq I, \ 0 \leq \epsilon \leq \delta\}, \ \delta \geq 0 .
\]
We simply denote by $\omega_{2}(f ; \delta):=\omega_{2}(f, \mathbb{R} ; \delta)$. It is well known (cf. \cite{adsaup}) that
\begin{equation}
	\omega_{2}(f , I; a \delta)\leq\lceil a\rceil^{2} \omega_{2}(f , I ; \delta) \leq(1+a)^{2} \omega_{2}(f , I ; \delta),\quad a, \delta \geq 0 , \label{D}
\end{equation}
where $\lceil a\rceil$ stands for the ceiling of $a$.  For any $x \in \mathbb{R}$, and $f \in {\cal C}_{S}(\mathbb{R})$, we write \begin{equation}
	f_{x}(y)=\left(f(y)-f(x+) \right)1_{(x, \infty)}(y)+\left(f(y) - f(x-)\right) 1_{(-\infty, x)}(y), \quad y \in \mathbb{R}, \label{D2}
\end{equation}
where $1_{A}$ stands for the indicator function of the set $A$. Note that $f_{x}$ is continuous at $x$ and $f_{x}(x)=0$.

The second interesting property of Steklov operators, and more generally of the operators $D_{m, k, h}$, is that they approximate functions $f \in {\cal C}_{S}(\mathbb{R})$, as shown is the following result.  If $g$ is a real function defined on $I \subseteq\mathbb{R}$, we denote $\|g\|_I=\sup \{|g(y)|: \ y \in I \}$.  If $I=\mathbb{R}$, we simply write $\|g\|:=\|g\|_\mathbb{R}$.
\begin{prop}Let $f \in {\cal C}_{S}(\mathbb{R})$ and $k \in \mathbb{N}_{0}$, with $k < m$. Then,
	\begin{equation}
		\left\|D_{m, k, h}(f ; x)-f(x)\right\| \leq c_{m, k}\ \omega_{2}(f ; h) \label{F}
	\end{equation}
	and
	\begin{equation}
		\left|D_{m, k, h}(f ; x)-f_{\star}(x)\right| \leq c_{m, k} \ \omega_{2}\left(f_{x},[x - m h, x + m h] ; h\right),\label{G}
	\end{equation}
	where
\begin{equation}
	c_{m, k}=	\frac{1}{2} \widetilde{\mathbb{E}}\left(\left\lceil \left|\widetilde{T}_{k}+\widetilde{S}_{m- k}\right|\right\rceil\right)^{2}\leq\frac{1}{2}\left(1+\sqrt{\frac{k^{2}}{2 k-1}+\frac{m-k}{3}}\right)^{2}. \label{cmk}
\end{equation}
	 \label{prbost}
\end{prop}
\begin{proof} We first prove inequality \rf{cmk}. We have from (\ref{A})
	\begin{equation}
		\widetilde{\mathbb{E}} \widetilde{R}_{k}^{2}=\frac{k^{2}}{\binom{2k}{k}} \sum_{j=1}^{k}\binom{k-1}{j-1}^{2}=\frac{k^{3}}{2(2 k-1)} . \label{I1}
	\end{equation}
	Since the random variables $k - \widetilde{R}_{k}$ and $\widetilde{R}_{k}$ have the same law, we see that $\widetilde{\mathbb{E}} \widetilde{R}_{k}=k / 2$. We thus have from (\ref{A}) and (\ref{I1})
	\[
	\widetilde{\mathbb{E}} \widetilde{T}_{k}^{2}=4 \widetilde{\mathbb{E}} \widetilde{R}_{k}^{2}-4 k \widetilde{\mathbb{E}} \widetilde{R}_{k}+k^{2}=\frac{k^{2}}{2 k-1},
	\]
	which implies that
	\[
	\widetilde{\mathbb{E}}\left(\widetilde{T}_{k}+\widetilde{S}_{m - k}\right)^{2}=\frac{k^{2}}{2 k-1}+\frac{m - k}{3} .
	\]
	This, Schwarz's inequality and the fact that  $\lceil a\rceil\leq 1+a,\ a\geq 0$,  show (\ref{cmk}).
	
	Let $y \in \mathbb{R}$ and denote $\widetilde{W}=\widetilde{T}_{k}+\widetilde{S}_{m-k}$. Since $\widetilde{W}$ is symmetric, we have from (\ref{D})
	\begin{align}
		&	|\widetilde{\mathbb{E}} f(y+h \widetilde{W})-f(y)|=\frac{1}{2}|\widetilde{\mathbb{E}}(f(y+h \widetilde{W})-2 f(y)+f(y-h \widetilde{W}))| \nonumber\\
		&\leq\frac{1}{2} \widetilde{\mathbb{E}} \omega_{2}\left(f; h|\widetilde{W}|\right) \leq \frac{1}{2} \widetilde{\mathbb{E}}\lceil|\widetilde{W}|\rceil ^{2}\omega_{2}\left(f ; h\right).\label{J}\end{align}
This shows (\ref{F}).  On the other hand, let $x \in \mathbb{R}$. By considering the cases $y>x$, $y = x$, and $y<x$, we have from \rf{D1} and \rf{D2}
\begin{align*}
	&f(y)-f_{\star}(x)-f_{x}(y)\\&=(f(x+)-f(y))\left(1_{(x, \infty)}(y)-\frac{1}{2}\right)+(f(x-)-f(y))\left(1_{(-\infty, x)}(y)-\frac{1}{2}\right), \quad y \in \mathbb{R}.
\end{align*}
We replace $y$ by $x+h \widetilde{W}$ in this identity and then take expectations. Since $\widetilde{W}$ is symmetric and absolutely continuous, we get
\[\widetilde{\mathbb{E}} f(x+h \widetilde{W})-f_{\star}(x)=\widetilde{\mathbb{E}} f_{x}(x+h \widetilde{W})=\widetilde{\mathbb{E}} f_{x}(x+h \widetilde{W})-f_{x}(x).\]
Thus, (\ref{G}) follows as in (\ref{J}) by taking into account that $|\widetilde{W}| \leq m$. \end{proof}
	\begin{rem} Direct computations show that
		\[c_{1,0}=\frac{1}{2}\quad \hbox{ and } \quad c_{2,0}=\frac{7}{8}.\]
	\label{re14}
	\end{rem}
\begin{rem}	Let $f \in {\cal C}_{S}(\mathbb{R})$. If $f$ is uniformly continuous, then $\omega_{2}(f ; h) \rightarrow 0$, as $h \rightarrow 0$.  Also, it follows from (\ref{dfsets}) and (\ref{D2}) that
\[\omega_{2}\left(f_{x},[x - m h, x+m h] ; h\right) \rightarrow 0, \quad h \rightarrow 0 . \]\label{re15}
\end{rem}

We point out that property (\ref{G}) and Remark \ref{re15} will play a crucial role to give confidence intervals for distribution functions and their derivatives having isolated discontinuities of the first kind.
		\section{Random Weierstrass-type operators} \label{serand}
	Roughly speaking, if in formula (\ref{dfweie}) we replace the function $f\in {\cal C}_S(\mathbb{R})$ by a bounded random function independent of $\widetilde{W}$, we obtain a random Weierstrass-type operator. More precisely, let $(\Omega,{\cal F}, P)$ be a probability space. Expectations and random variables defined on this space are denoted by $\mathbb{E}$ and $Y$, respectively.  We consider the product  $(\Omega\times \widetilde{\Omega},{\cal F}\otimes \widetilde{\cal F}, P\times \widetilde{P})$. Expectations on this space are denoted by $\mathbb{E}\times\widetilde{\mathbb{E}}$.
	Suppose that $Y$ (resp. $\widetilde{W}$) is a real-valued random variable defined on $\Omega$ (resp. $\widetilde{\Omega}$).   We can always assume that these two random variables are defined on  $\Omega\times \widetilde{\Omega}$ via the formulas
	\[	
	Y_1(\omega,\widetilde{\omega})=Y(\omega),\quad \widetilde{W}_1(\omega,\widetilde{\omega})=\widetilde{W}( \widetilde{\omega}),\quad (\omega,\widetilde{\omega})\in\Omega\times\widetilde{\Omega}.
	\]
	It is easily seen that 	$Y_1$ and $\widetilde{W}_1$ are independent.
	
	We always consider bounded stochastic processes $\mathbb{Y}=\left(Y(x),\ x\in \mathbb{R}\right)$ defined on $\Omega$ such that the map $(x,\omega) \rightarrow Y(x)(\omega),\ x\in \mathbb{R},\ \omega \in \Omega  $, is jointly measurable.  This is the case, for instance, if $\mathbb{Y}$ has right-continuous paths.  In such a case, $Y(x+h\widetilde{W})$ is a random variable defined on $\Omega\times \widetilde{\Omega}$.
	
	Let $\widetilde{W}$ be as in (\ref{dfweie}) and let $\mathbb{Y}_n=\left(Y_n(x),\ x\in \mathbb{R}\right)$ be a bounded stochastic process defined on $\Omega$.  The random Weierstrass-type operator associated to $\mathbb{Y}_n$ is defined as
	\begin{equation}
		L_h( Y_n;x)=\widetilde{\mathbb{E}}Y_n(x+h\widetilde{W}).\label{dfrast}
	\end{equation}
	The motivation of this definition comes from the following estimation problem.  Suppose that $\mathbb{E}Y_n(x)=\mu(x)$, where $\mu\in {\cal C}_S(\mathbb{R})$ is unknown, as well as,
\begin{equation}
	\lim _{n \rightarrow \infty} \mathbb{E}\left(Y_{n}(x)-\mu(x)\right)^{2}=0, \quad x \in \mathbb{R} . \label{K}
\end{equation}
Instead of using the process $\mathbb{Y}_n$ as an estimator of $\mu$, we consider the identity
			\begin{equation}
			L_h(Y_n;x)-\mu_{\star}(x)=L_h(Y_n-\mu;x)+L_h(\mu;x)-\mu_{\star}(x),\label{decomp}
		\end{equation}
	where $\mu_{\star}$ is defined in (\ref{D1}).  The random term on the right-hand side in (\ref{decomp}) has zero $\mathbb{E}$-expectation, as follows from \rf{dfrast} and Fubini's theorem.
	
	Moreover, we see from Schwarz's inequality and Fubini's theorem that
	\begin{align*}
		&\lim _{n \rightarrow \infty} \mathbb{E}\left(L_{h}\left(Y_{n}-\mu ; x\right)\right)^{2}=\lim _{n \rightarrow \infty} \mathbb{E}\left(\widetilde{\mathbb{E}}\left(Y_{n}(x+h \widetilde{W})-\mu(x+h \widetilde{W})\right)\right)^{2}\\&\leq\lim _{n \rightarrow \infty} \widetilde{\mathbb{E}} \times \mathbb{E}\left(Y_{n}(x+h \widetilde{W}) - \mu(x+h \widetilde{W})\right)^{2}=0,\end{align*} where the last equality follows from (\ref{K}) and the dominated convergence theorem.

The deterministic term on the right-hand side in (\ref{decomp}) converges to zero, as $h \rightarrow 0$, provided that $\mu \in {\cal C}(\mathbb{R})$. More generally, if $\mu \in {\cal C}_{S}(\mathbb{R})$ the same property is true, whenever $\widetilde{W}$ is symmetric and absolutely continuous. The proof is similar to that given in Proposition \ref{prbost}.  If the operator $L_h$ takes non-smooth into smooth functions, identity (\ref{decomp}) implies that we can estimate $\mu_{\star}(x)$ by the smooth approximant $L_h(Y_n;x)$.  If, in addition, $\mu \in {\cal C}_{S}^{k}(\mathbb{R})$, we can differentiate in (\ref{decomp}) in order to estimate the derivatives of $\mu$ via the approximation
	\begin{equation}
		L_h^{(k)}(Y_n;x)-(\mu^{(k)})_{\star}(x)=L_h^{(k)}(Y_n-\mu;x)+L_h^{(k)}(\mu;x)-(\mu^{(k)})_{\star}(x),\quad k\in \mathbb{N}.\label{decomd}
	\end{equation}
	
	In this paper, attention is focused on the case in which $\mu=F\in {\cal C}_S(\mathbb{R})$ is the distribution function of a certain random variable $X$ and $\mathbb{Y}_n=\hat{\mathbb{F}}_n=(\hat{\mathbb{F}}_n(x), \ x\in \mathbb{R})$ is its corresponding empirical process. In this setting, similar formulas to identity (\ref{decomd}), with $(\mu^{(k)})_{\star}(x)$ replaced by $\mu^{(k)}(x)$, are widely used (see Schuster \cite{scesti}, Silverman \cite{siweak}, and Scott \cite{scott}, among others).  Note that, in accordance with the usual notation for the empirical process, we omit in this case the rule for notation of random variables defined on $\Omega$ stated at the beginning of this section. Specifically, let $\sequ{X}{j}$ be a sequence of independent copies of $X$ and define
		\begin{equation}
		\hat{F}_n(x)=\frac{1}{n}\sum_{j=1}^{n}1_{(-\infty,x]}(X_j), \quad x\in \mathbb{R}. \label{pampiri}
	\end{equation}
	Alternatively (cf. Shorack and Wellner \cite{shweem}), $\hat{\mathbb{F}}_n$ can be rewritten as follows. Let $\left(U_j\right)_{j\geq 1}$ be a sequence of independent copies of a random variable $U$ uniformly distributed on $[0,1]$ and denote
\[	
		S_n(x)=\sum_{j=1}^{n}1_{[0,x]}(U_j), \quad x\in [0,1].\]
Then,
	\begin{equation}
	\left(	\hat{F}_n(x), \  x\in \mathbb{R}   \right)\eqla \left(\frac{S_n(F(x))}{n},\ x\in \mathbb{R}\right), \label{eqlaym}
	\end{equation}
where $\eqla$ stands for equality in law.

Suppose that $\widetilde{W}$ is a symmetric random variable with distribution function $\widetilde{F}$ and probability density $\widetilde{\rho}$. Then, it follows from (\ref{pampiri}) that \begin{equation}
L_{h}\left(\hat{F}_{n} ; x\right)=\frac{1}{n} \sum_{j=1}^{n} \widetilde{\mathbb{E}} 1_{(-\infty, x+h \widetilde{W}]}\left(X_{j}\right)=\frac{1}{n} \sum_{j=1}^{n} \widetilde{F}\left(\frac{x-X_{j}}{h}\right) \label{M}	
\end{equation}
is the estimator of $F_{\star}$. If, in addition, $X$ is absolutely continuous with density $\rho$, we can differentiate in (\ref{M}) to obtain
\begin{equation}
	L_{h}^{(1)}\left(\hat{F}_{n} ; x\right)=\frac{1}{n h} \sum_{j=1}^{n} \widetilde{\rho}\left(\frac{x-X_{j}}{h}\right), \label{star}
\end{equation}
which is the kernel estimator of $\rho_\star$.  In this way, classical kernel estimators of second order are interpreted in terms of the first derivative of random Weierstrass-type operators.  A different interpretation in terms of convolutions can be found in Tenkorang et al. \cite{tekern}.

As said in the Introduction, the aim of this paper is to obtain explicit confidence bands and intervals for distribution functions and their derivatives. For this reason, we will not consider general random operators as those in (\ref{dfrast}), but random Steklov operators. Indeed, recalling (\ref{dfstek}) and (\ref{eqlaym}), we define the random Steklov operator of order $m$ as
\begin{equation}
	L_{m, h}\left(\hat{F}_{n} ; x\right)=\widetilde{\mathbb{E}} \hat{F}_{n}\left(x+h \widetilde{S}_{m}\right)=\frac{1}{n} \widetilde{\mathbb{E}} S_{n}\left(F\left(x+h \widetilde{S}_{m}\right)\right) . \label{24}
\end{equation}
The following proposition gives explicit expressions for $L_{m,h}^{(k)}(\hat{F}_n;x)$.
	\begin{prop}
		For any $k=0,\dots m-1$, we have
		\[	L_{m,h}^{(k)}(\hat{F}_n;x)=\frac{1}{(2h)^k}\Delta_h^k	L_{m-k,h}(\hat{F}_n;x)=\frac{1}{nh^k}\sum_{j=1}^{n}\widetilde{F}_m^{(k)}\left(\frac{x-X_j}{h}\right),\]
		where $\widetilde{F}_m$ is the distribution function of $\widetilde{S}_m$. \label{prderast}
	\end{prop}
\begin{proof} The first equality readily follows from (\ref{stekde}).  The second one follows as in (\ref{star}).
\end{proof}
The first equality in Proposition \ref{prderast} is interesting for theoretical purposes, whereas the second one is useful for computations and simulations.

\section{Confidence bands for $F^{(k)}$}\label{seconf}
From now on, we fix $0<\alpha<1$, $1-\alpha$ being the confidence level. Given two sequences $\left(x_{k}\right)_{k \geq 1}$ and $\left(y_{k}\right)_{k \geq 1}$, we set $x_{k} \sim y_{k}$ if
\begin{equation}
-\infty<\underset{k \rightarrow \infty}{\underline{\lim}} \frac{x_{k}}{y_{k}} \leq \underset{k \rightarrow \infty}{\overline{\lim}} \frac{x_{k}}{y_{k}} < \infty . \label{deriem}
\end{equation}
To give confidence bands for $F^{(k)}$, we use the celebrated Dvoretzky-Kiefer-Wolfowitz inequality (cf. \cite{dvkias}) in the final form shown by Massart \cite{mathet}, which establishes that
\begin{equation}
	P\left(\left\|\frac{S_{n}(x)}{n} - x\right\|_{[0,1]}>\delta\right) \leq 2 e^{-2 n \delta^{2}}, \quad \delta>0 . \label{40}
\end{equation}
Note that the parameter $m$ in the Steklov operator $L_{m, h}$ only provides smoothness, as seen in Proposition \ref{prstde}. However, as far as rates of convergence are concerned, we should choose $m$ as small as possible, as follows from Proposition \ref{prbost}, particularly from the coefficients $c_{m,k}$ defined in (\ref{cmk}). For these reasons, we choose $m=k+1$ in order to estimate $(F^{(k)})_{\star}$. In other words, let us replace $L_{h}$ by $L_{k+1, h}$ and $\mu$ by $F$ in identity (\ref{decomd}). Choose $\mathbb{Y}_{n}=\hat{\mathbb{F}}_{n}$ as in (\ref{eqlaym}), and define the random variable
\begin{equation}
	T_{k,n}=\frac{\Delta_{h}^{k} L_{1, h}\left(\hat{F}_{n}-F;x\right)}{(2 h)^{k}}=\frac{1}{(2 h)^{k}} \widetilde{\mathbb{E}} \Delta_{h}^{k}\left(\frac{S_{n}\left(F\left(x+h \widetilde{S}_{1}\right)\right)}{n}- F\left(x+h \widetilde{S}_{1}\right)\right). \label{dftkh}
\end{equation}
 With the previous choices, we obtain the basic identity
\begin{align}
	& L_{k+1,h}^{(k)}\left(\hat{F}_{n} ; x\right)-(F^{(k)})_{\star}(x)=L_{k+1,h}^{(k)}\left(\hat{F}_{n}-F ; x\right)+L_{k+1, h}^{(k)}(F ; x)-(F^{(k)})_{\star}(x) \nonumber \\
	& =	T_{k,n}+L_{k+1, h}\left(F^{(k)} ; x\right)-(F^{(k)})_{\star}(x), \quad F \in {\cal C}_{S}^{k}(\mathbb{R}),\quad k \in \mathbb{N}_{0}, \label{31v3}
\end{align}
where we have used (\ref{stekde}) and  (\ref{stekd2}). Note that the dependence on $x$ and $h$ of the random variable $T_{k, n}$ has been dropped.

Let $k \in \mathbb{N}_{0}$. From Proposition \ref{prderast}, the stochastic process

\begin{equation*}
	\left(L_{k+1,h}^{(k)}\left(\hat{F}_{n} ;x\right)-F^{(k)}(x), \quad x \in \mathbb{R}\right)
\end{equation*}
has continuous paths if $F^{(k)}$ is continuous. Therefore, its supremum norm is a random variable. Denote
\begin{equation}
	t_{\alpha}=\sqrt{2 \log (2 / \alpha)}. \label{41}
	\end{equation}
	The first main result of this paper is the following.
\begin{thm} Let $k \in \mathbb{N}_{0}$ and assume that $F^{(k)}$ is uniformly continuous. Let $b_{k}(h)$, $h \geq 0$, be a nondecreasing continuous function such that $b_{k}(0)=0$ and
\begin{equation}
	\left\|L_{k+1, h}\left(F^{(k)} ; x\right) - F^{(k)}(x)\right\| \leq b_{k}(h) . \label{b}
\end{equation}
Then,
\begin{equation}
	P\left(\left\|L_{k+1, h}^{(k)}\left(\hat{F}_{n} ; x\right) - F^{(k)}(x)\right\|> \frac{t_{\alpha}}{2 h^{k}\sqrt{n} }+ b_{k}(h)\right) \leq \alpha . \label{c3}
\end{equation}
\label{teban1}
\end{thm}
\begin{proof} From (\ref{dfdire}) and  (\ref{dftkh}), we see that
\begin{equation}
		\left\|T_{k,n}\right\|\leq \frac{1}{h^{k}}\left\|\frac{S_{n}(x)}{n}-x\right\|_{[0,1]}.\label{As}
\end{equation}
We therefore have from (\ref{40}), (\ref{31v3}) and (\ref{b})
	\begin{align*}
	&	P\left(\left\|L_{k+1, h}^{(k)}\left(\hat{F}_{n} ; x\right) -F^{(k)}(x)\right\| >\frac{t_{\alpha}}{2 h^{k}\sqrt{n} }+ b_{k}(h)\right) \leq P\left(\left\|T_{k,n}\right\|>\frac{t_{\alpha}}{2 h^{k}\sqrt{n} }\right) \\\leq &P\left(\left\|\frac{S_{n}(x)}{n} - x\right\|_{[0,1]}>\frac{t_{\alpha}}{2\sqrt{n} }\right)\leq  2 e^{-t_\alpha^{2}/2}=\alpha,
	\end{align*}
where the last equality follows from (\ref{41}).  The proof is complete.
\end{proof}
The assumption on uniform continuity in Theorem \ref{teban1} cannot be weakened.  In fact, Schuster \cite{scesti} showed that a large class of kernel estimators of $F^{(1)}$ uniformly converge to $F^{(1)}$ with probability one if and only if $F^{(1)}$ is uniformly continuous.  See Silverman \cite{siweak} for related results concerning $F^{(k)},\ k\in \mathbb{N}$.

Let $k=0$. By (\ref{eqlaym}) and (\ref{40}), we have for the empirical process $\hat{\mathbb{F}}_{n}$
\begin{equation}
	P\left(\left\|\hat{F}_{n}(x)-F(x)\right\|>\frac{t_{\alpha}}{2 \sqrt{n}}\right) \leq \alpha . \label{A3}
\end{equation}
As follows from (\ref{b}) and (\ref{c3}), we can choose the bandwidth $h$ in $b_{0}(h)$ as small as we wish. Thus, the estimators $L_{1,h}(\hat{F}_{n} ; x)$ and $\hat{F}_{n}(x)$ have essentially the same behavior. The difference is that $L_{1,h}(\hat{F}_{n} ; x)$ has continuous paths, whereas $\hat{F}_{n}(x)$ has not.

Wang et al. \cite{wachsm} obtained an asymptotic result stating that a continuous confidence band of length $l\sim n^{-1 / 2}$ can be constructed whenever $F$ on $F^{(1)}$ satisfy a Hölder condition. Using sup-functionals of weighted empirical processes, Stepanova and Pavlenko \cite{stpago} and Dümbgen and Wellner \cite{duwean} constructed asymptotic confidence bands for $F$, much shorter in the tails of $F$ than those in (\ref{c3}) for $k=0$. In this regard, Maillard \cite{maloca} obtained exact formulas for the tail probabilities
\[
P\left(\left\|\frac{S_{n}(x)}{n} - x\right\|_{[a, b]}>\delta\right), \quad 0 \leq a \leq b \leq 1 .
\]
Using these formulas for the intervals $[0,b]$ and $[a, 1]$ and proceeding as in the proof of Theorem \ref{teban1}, it seems possible to improve this theorem, particularly for the tails of $F$. However, this will not pursued here.

Suppose that $k \in \mathbb{N}$. By  (\ref{C}), Proposition \ref{prbost} with $m=k+1$, and (\ref{b}), the function $b_{k}(h)$ can be chosen as
\begin{equation}
	b_{k}(h)=e_{k+1} \omega_{2}\left(F^{(k)} ; h\right), \quad c_{k+1,0} \leq e_{k+1}:=\frac{1}{2}\left(1+\sqrt{\frac{k+1}{3}}\right)^{2} . \label{B3}
\end{equation}
To specify the lengths of the confidence bands in (\ref{c3}), assume that $F^{(k)}$ satisfies the Lipschitz condition
\begin{equation}
	\omega_{2}\left(F^{(k)} ; h\right) \leq A h^{\beta},  \quad  0 <\beta \leq 2, \label{C3r}
\end{equation}
for some positive constant $A$. In such a case, we state the following result.
\begin{cor} Let $k \in \mathbb{N}$ and assume that $F^{(k)}$ satisfies (\ref{C3r}). Then,
\begin{equation}
	P\left(\left\|L_{k+1, h}^{(k)}\left(\hat{F}_{n} ; x\right)-F^{(k)}(x)\right\|>\frac{C_{k}(\alpha)}{n^{\beta /(2 \beta+2 k)}}\right) \leq \alpha, \label{D3}
\end{equation}
where
\begin{equation}
	C_{k}(\alpha)=\frac{t_{\alpha}}{2} D_{k}(\alpha)^{-k /(\beta+k)}+A e_{k+1} D_{k}(\alpha)^{\beta /(\beta+k)}, \quad D_{k}(\alpha)=\frac{k t_{\alpha}}{2 \beta A e_{k+1}}, \label{E3}
\end{equation}
and
\begin{equation}
	h=D_{k}(\alpha)^{1 /(\beta+k)} n^{-1 /(2 \beta+2 k)} . \label{F3}
\end{equation}\label{co32}
\end{cor}

\begin{proof} By (\ref{B3}) and (\ref{C3r}), we have from Theorem \ref{teban1}
\begin{equation}
	P\left(\left\|L_{k+1}^{(k)}\left(\hat{F}_{n} ; x\right)-{F}^{(k)}(x)\right\|>\frac{t_{\alpha}}{2 h^{k} \sqrt{n} }+A e_{k+1} h^{\beta}\right) \leq \alpha . \label{G3}
\end{equation}
The function
\begin{equation*}
	f(h)=\frac{t_{\alpha}}{2  h^{k} \sqrt{n}}+A e_{k+1} h^{\beta}, \quad h \geq 0,
\end{equation*}
attains its minimum at the value of $h$ given in (\ref{F3}). Thus, the result follows from (\ref{G3}) and some simple computations.
\end{proof}
The length of the confidence band in (\ref{D3}) strongly depends on the degree of smoothness of $F^{(k)}$, determined by the parameter $\beta$. On the other hand, it follows from (\ref{B3}) and (\ref{E3}) that $D_{k}(\alpha) \sim 1$ and $C_{k}(\alpha) \sim k$, as $k \rightarrow \infty$. Therefore, keeping constant the degree of smoothness of $F^{(k)}$, the length of the confidence band in (\ref{D3}) increases as $k$ increases.

Let $k=1$. Giné and Nickl \cite{ginico} obtained (asymptotic) globally adapted confidence bands in finite intervals for a probability density $\rho= F^{(1)}$ under the so called similarity conditions on  $\rho$. Patschkowski and Rohde \cite{parolo} extended the previous results to locally adaptive confidence bands for $\rho$. Roughly speaking, the length $l$ of the aforementioned bands is $l\sim(\log n)^{\gamma} n^{-\beta /(2 \beta+1)}$, for some $\gamma \geq \beta /(2 \beta+1)$, whereas that in Corollary \ref{co32} is $l\sim n^{-\beta /(2 \beta+2)}$.  For $k \in \mathbb{N}$, Schuster \cite{scesti} obtained similar results to Theorem \ref{teban1} for general kernel estimators of second order.  For the kernel estimators considered in this paper, Theorem \ref{teban1} improves such results.

It is worth mentioning the following special case.

\begin{cor} Let $k \in \mathbb{N}$ and $a, b \in \mathbb{R}$ with $a<b$. Assume that $F$ is a polynomial of degree $k+1$ at most on the interval $\left[a-(k+1) h_{0}, b+(k+1) h_{0}\right]$, for some $h_{0}>0$. Then,
\begin{equation*}
	P\left(\left\|L_{k+1, h_{0}}^{(k)}\left(\hat{F}_{n} ; x\right) - F^{(k)}(x)\right\|_{[a, b]}>\frac{t_{\alpha}}{2 h_{0}^{k} \sqrt{n}}\right) \leq \alpha .
\end{equation*} \label{co33}
\end{cor}
\begin{proof} Since the operator $L_{k+1, h}$ preserves affine functions, we choose $h=h_{0}$ in (\ref{31v3}) to obtain
\begin{equation}
	L_{k+1, h_{0}}^{(k)}\left(\hat{F}_{n} ; x\right)-F^{(k)}(x)=T_{k, n}, \quad x \in[a, b] . \label{H3}
\end{equation}
With the choice $h=h_{0}$ in (\ref{dftkh}), we have from (\ref{dfdire}), (\ref{As}), and (\ref{H3})

\begin{equation*}
	\left\|L_{k+1, h_{0}}^{(k)}\left(\hat{F}_{n} ; x\right)- F^{(k)}(x)\right\|_{[a, b]} \leq \frac{1}{h_{0}^{k}}\left\|\frac{S_{n}(x)}{n}- x\right\|_{[0,1]} .
\end{equation*}
Therefore, the result follows as in the proof of Theorem \ref{teban1}.
\end{proof}
Observe that the length of the confidence band in $[a, b]$ in Corollary \ref{co33} is of order $n^{-1 / 2}$. For $k=1$, in the context of log-concave density estimation, an asymptotic result similar to Corollary \ref{co33} was obtained by Walther et al. \cite{walt}.
\section{Confidence intervals} \label{seinte}
In this section, $\left(\epsilon_{n}\right)_{n \geq 1}$ and $\left(\nu_{n}\right)_{n \geq 1}$ are arbitrary sequences of nonnegative real numbers converging to $0$, as $n \rightarrow \infty$. On the other hand, denote by $\sigma(y)=\sqrt{y(1-y)},\ 0\leq y\leq 1$. Finally, consider the nonnegative, increasing, and convex function
\begin{equation}
	\tau(\delta)=(1+\delta) \log (1+\delta)-\delta, \quad \delta \geq 0 . \label{a4}
\end{equation}
We start by giving confidence intervals for $F_{\star}$.
\begin{thm}Let $F \in {\cal C}_{S}(\mathbb{R})$ and assume that $\left(\sigma^{2} \circ F\right)_{\star}(x)>0$. Choose the bandwidth $h=h(x, n)$ in such a way that
\begin{equation}
	\left|L_{1, h}(F ; x)-F_{\star}(x)\right| \leq \sqrt{\frac{\epsilon_{n}}{n}}, \quad \left|L_{1, h}\left(\sigma^{2} \circ F ; x\right)-\left(\sigma^{2} \circ F\right)_{\star}(x)\right| \leq \nu_{n} . \label{b4}
\end{equation}
Then,
\begin{equation}
	P\left(\left|L_{1, h}\left(\hat{F}_{n} ; x\right)-F_{\star}(x)\right| \geq l_{0}(x,n)\right)\leq \alpha, \label{c4}
\end{equation}
where
\begin{equation}
l_{0}(x,n)=\left(\left(\sigma^{2} \circ F\right)_{\star}(x)+\nu_{n}\right) \tau^{-1}\left(\frac{\log (2 / \alpha)}{n\left(\left(\sigma^{2} \circ F\right)_{\star}(x)+\nu_{n}\right)}\right)+\sqrt{\frac{\epsilon_{n}}{n}} . \label{d4}
\end{equation}\label{te41}
\end{thm}

\begin{rem} By (\ref{a4}), $ \tau(\delta) \sim \delta^{2} / 2$, as $\delta \rightarrow 0$, thus implying that $\tau^{-1}(\delta) \sim \sqrt{2 \delta}$, as $\delta\rightarrow 0$. This means that the length of the confidence interval in (\ref{d4}) satisfies \begin{equation}
l_{0}(x,n) \sim \frac{1}{\sqrt{n}}\left(t_{\alpha} \sqrt{\left(\sigma^{2} \circ F\right)_{\star}(x)+\nu_{n}}+\sqrt{\epsilon_{n}}\right) \sim \frac{t_{\alpha}}{\sqrt{n}}(\sigma \circ F)_{\star}(x), \quad n \rightarrow \infty, \label{e4}
\end{equation}
where $t_{\alpha}$ is defined in (\ref{41}). For the usual confidence level $1-\alpha=0.95$, we have $t_{\alpha}=2.716 \cdots$. \label{rem42} \end{rem}

From Proposition \ref{prbost} and Remark \ref{re14}, the first inequality in (\ref{b4}) is fulfilled if
\[
 \frac{1}{2}\ \omega_{2}\left(F_{x} ,[x-h, x+h] ; h\right) \leq \sqrt{\frac{\epsilon_{n}}{n}} .
\]
A similar statement holds for the second inequality in (\ref{b4}). Thus, we can determine the bandwidth $h$ by estimating the local second modulus of the functions under consideration. In this respect, recall the comments in Remark \ref{re15}.

On the other hand, suppose that $F \in {\cal C}(\mathbb{R})$. Since $F$ is uniformly continuous, the corresponding length of the confidence band is, asymptotically, $t_{\alpha} / (2\sqrt{n})$, as follows from Theorem \ref{teban1} and the comments following it.  Since $\left(\sigma^{2} \circ F\right)_{\star}(x)\leq\frac{1}{2}$, we see that the length $l_{0}(x,n)$ in (\ref{e4}) is asymptotically smaller than  $t_{\alpha} / (2\sqrt{n})$, specially for those values of $x$ in the tails of $F$. However, both lengths have the same order of magnitude.

In order to give confidence intervals when $\left(\sigma^{2} \circ F\right)_{\star}(x)=0$, we will give closed form expressions, as well as accurate estimates, for the mean standard error (MSE) of $L_{1, h}(\hat{F}_{n} ; x)$. In this respect, denote by $a \wedge b=\min (a, b)$ and $a \vee b=\max (a, b), a, b \in \mathbb{R}$.
\begin{thm}
	Let $\widetilde{R}_{1}$ be an independent copy of $\widetilde{S}_{1}$. Then,
	\begin{align}
		&MSE\left(L_{1, h}\left(\hat{F}_{n} ; x\right)\right) \nonumber\\
&=\frac{1}{n} \widetilde{\mathbb{E}} F\left(x+h\left(\widetilde{R}_{1} \wedge \widetilde{S}_{1}\right)\right)\left(1-F\left(x+h\left(\widetilde{R}_{1} \vee \widetilde{S}_{1}\right)\right)\right)
		+\left(L_{1, h}(F ; x)-F_{\star}(x)\right)^{2} \nonumber\\
& \leq \frac{1}{n}\left(\sigma^{2}\circ F\right)_ {\star}(x)+ \frac{1}{n} \left(L_{1, h}\left(\sigma^{2}\circ F ; x\right)-\left(\sigma^{2}\circ F\right)_ {\star}(x)\right)+\left(L_{1, h}(F;x)-F_{\star}(x)\right)^{2}.\label{64a}
	\end{align}\label{pra3}
\end{thm}
If we choose the bandwidth $h$ as in (\ref{b4}) and $(\sigma^{2}\circ F)_ {\star}(x)=0$, then the right-hand side in (\ref{64a}) is very small. As a consequence, we obtain the following result.

\begin{thm} Let $F \in  {\cal C}_{S}(\mathbb{R})$ and assume that $\left(\sigma^{2} \circ F\right)_{\star}(x)=0$. Choose the bandwidth $h$ as in (\ref{b4}). Then,
\begin{equation}
	P\left(\left|L_{1, h}\left(\hat{F}_{n} ; x\right)-F_{\star}(x)\right| \geq \sqrt{\frac{\epsilon_{n}+\nu_{n}}{\alpha n}}\right) \leq \alpha . \label{f4}
\end{equation} \label{te43}
\end{thm}
The length of the confidence interval in (\ref{f4}) is as short as we wish, provided we choose an appropriately small bandwidth $h$. In this regard, suppose that $F \in {\cal C}(\mathbb{R})$ is concentrated on some finite interval $[a, b]$. Theorem \ref{te43} applies when $x=a$ or $x=b$. It is well known (cf. \cite{adalra,baappl,leones}) that Bernstein estimators match $F$ at $x=a$ and $x=b$. In comparison with (\ref{f4}), Bernstein estimators provide a better fit at the endpoints than kernel estimators.  However, this disadvantage can be overcomed with a small bandwidth.

Assume that $F \in {\cal C}^{2}(\mathbb{R})$ is concentrated on a finite interval $[a, b]$ and that $F^{(1)}(a+) \neq 0$ and $F^{(1)}(b -) \neq 0$. Zhang et al. \cite{zhesti} showed that, for the points in the intervals $[a, a+h)$ and $(b-h, b]$, the bias of the kernel distribution estimator converges to zero at the rate $O(h)$, which is slower than the usual rate $O\left(h^{2}\right)$ for distribution functions $F$ not concentrated on a finite interval $[a, b]$ (cf. Azzalini \cite{azanot} and Jones \cite{jothep}). This fact, however, has no influence in the length of the confidence intervals in Theorems \ref{te41} and \ref{te43}, since in both results we choose a very small bandwidth $h$.

An application to discrete random variables is given in the following result.
\begin{cor} Let $X$ be a random variable taking values in $S=\left\{x_{j}, \ j \in \mathbb{Z}\right\}$, as defined in (\ref{dfsets}). Assume that $\left(\sigma^{2} \circ F\right)_{\star}(x)>0$. Then,
\[
P\left(\left|L_{1, h}\left(\hat{F}_{n} ; x\right)-F_{\star}(x)\right| \geq\left(\sigma^{2} \circ F\right)_{\star}(x) \tau^{-1}\left(\frac{\log (2 / \alpha)}{n\left(\sigma^{2} \circ F\right)_{\star}(x)}\right)\right) \leq \alpha,
\]
whenever the bandwidth $h$ is selected as follows:
\begin{enumerate}[(i)]
	\item if $x=x_{j}$, for some $j \in \mathbb{Z}$, then $h\leq s$, and
\item if $x \in\left(x_{j}, x_{j +1}\right)$, for some $j \in \mathbb{Z}$, then $h\leq \min\left(x-x_{j}, x_{j+1}-x\right)$.
\end{enumerate} \label{co44}
\end{cor}
Note that if the bandwidth is selected as in (i) and (ii), then condition (\ref{b4}) is fulfilled with $\epsilon_{n}=\nu_{n}=0, \ n \geq 1$. Therefore, the result is an immediate consequence of Theorem \ref{te41}. Observe that the length of the confidence interval in Corollary \ref{co44} is asymptotically the same as that in (\ref{e4}). An illustration of this result when $X$ has the Poisson distribution will be given in Section \ref{sesimu}.

Assume that $F \in {\cal C}_{S}^{1}(\mathbb{R})$ and set $\rho=F^{(1)}$. To give confidence intervals for $\rho_{\star}$. we start with  (\ref{dftkh}) and (\ref{31v3}) for $k=1$, that is,
\begin{equation}
	L_{2, h}^{(1)}\left(\hat{F}_{n} ; x\right)-\rho_{\star}(x)=T_{1, n}+L_{2, h}(\rho ; x)-\rho_{\star}(x) . \label{g4}
\end{equation}
\begin{thm} Let $r \geq 0$. Assume that $\rho \in {\cal C}_{S}(\mathbb{R})$ and that $\rho_{\star}(x)>0$. Choose the bandwidth $h=h(x, n)$ in such a way that
\begin{equation}
	\left|L_{2, h}(\rho ; x)-\rho_{\star}(x)\right| \leq \sqrt{\frac{r}{2 nh}} . \label{h4}
\end{equation}
Then,
\begin{equation}
	P\left(\left|L_{2, h}^{(1)}\left(\hat{F}_{n} ; x\right)-\rho_{\star}(x)\right| \geq l_{1}(x,n)\right) \leq \alpha, \label{i4}
\end{equation}
where

\begin{equation}
	l_{1}(x,n)=\left(\rho_{\star}(x)+\sqrt{\frac{r}{2 n h}}\right) \tau^{-1}\left(\frac{\log (2 / \alpha)}{2 n h\left(\rho_{\star}(x)+\sqrt{r/(2 n h)}\right)}\right)+ \sqrt{\frac{r}{2 n h}} . \label{j4}
\end{equation}\label{te45}
\end{thm}

\begin{rem} Using the same arguments as in Remark \ref{rem42}, we have
\begin{equation}
	l_{1}(x,n) \sim \frac{1}{\sqrt{2 n h}}\left(t_{\alpha} \sqrt{\rho_{\star}(x)}+\sqrt{r}\right), \quad n h \rightarrow \infty . \label{k4}
\end{equation}\label{rem92}\end{rem}

Let $x \in[a, b]$.  Suppose that $\rho$ is affine on $[a-2h_0, b+2h_0]$, for some $h_0>0$.  Choose $h=h_0$ in Theorem \ref{te45}. Since the operator $L_{2, h_0}$ preserves affine functions, inequality (\ref{h4}) is fulfilled for $r=0$. Therefore, the length in (\ref{j4}) takes on the form
\begin{equation}
	l_{1}(x,n)=\rho_{\star}(x) \tau^{-1}\left(\frac{\log (2 / \alpha)}{2 n h_{0} \rho_{\star}(x)}\right) \sim t_{\alpha} \sqrt{\frac{\rho_{\star}(x)}{2 h_{0}}} \frac{1}{\sqrt{n}}, \quad n \rightarrow \infty . \label{l4}
\end{equation}
In other words, the order of magnitude of the length in (\ref{l4}) is $1 / \sqrt{n}$.

Under the assumptions that $\rho$ has at least four continuous derivatives in a neighborhood of $x$ and $\rho(x)>0$, Hall \cite{hall} gave an asymptotic confidence interval for $\rho(x)$ whose length is of order $(nh)^{-1/2}$, stating that undersmoothing is preferable to explicit bias correction (see also Calonico et al. \cite{calo} for a comparison between the undersmoothing and the robust bias correction methods). Note that, in the setting of Theorem \ref{te45}, undersmoothing is achieved by choosing a small enough parameter $r$ in inequality (\ref{h4}).

Under a Lipschitz condition on the bias of the estimator at hand, we give the following specific result.

\begin{cor} Assume that $\rho \in {\cal C}_{S}(\mathbb{R})$, $\rho_{\star}(x)>0$, $H(x)>0$, and
\begin{equation}
|L_{2,h}(\rho;x)-\rho_{\star}(x)|\leq A(x)h^{\beta}, \quad 0<\beta\leq 2,\quad 0<h\leq H(x), \label{z1}
\end{equation}
for some $A(x)>0$. Whenever $h_0\leq H(x)$, we have
\begin{equation}
P\left(|L_{2,h_0}^{(1)}(\hat{F}_{n};x)-\rho_{\star}(x)|\geq l_1(x,n)\right)\leq \alpha, \label{z2}
\end{equation}
where
\begin{equation}
l_1(x,n) \sim \frac{2\beta+1}{2\beta}t_{\alpha}\sqrt{\rho_{\star}(x)}\frac{1}{\sqrt{2nh_0}}, \label{z3}
\end{equation}
and
\begin{equation}
h_0=\left(\frac{t_{\alpha}^2\rho_{\star}(x)}{8\beta^2A^2(x)n}\right)^{1/(2\beta+1)}. \label{z4}
\end{equation}\label{cor99}
\end{cor}

Observe that (\ref{z3}) and (\ref{z4}) give us
\begin{equation}
h_0 \sim n^{-1/(2\beta+1)},\quad and \quad l_1(x,n) \sim \left(\frac{\rho_{\star}(x)}{n}\right)^{\beta/(2\beta+1)}. \label{z44}
\end{equation}
On the other hand, condition (\ref{z1}) is fulfilled if
\begin{equation}
\frac{7}{8}\omega_2(\rho_x, [x-2h,x+2h])\leq A(x)h^{\beta},\quad 0<\beta\leq 2,\quad 0<h\leq H(x).\label{z5}
\end{equation}
This follows from (\ref{D2}), (\ref{G}), and Remark \ref{re14}.

In the case that $x$ is close to a discontinuity point of $\rho$, inequality (\ref{z1}) is fulfilled only for a small $H(x)$ and therefore we need a very large sample size $n$ to guarantee (\ref{z2}). This is illustrated in Section \ref{sesimu} when $\rho$ is the Pareto density and $x>0$ is close to the origin.

Finally, in order to compare Corollaries \ref{co32} and \ref{cor99}, suppose that $\rho$ is uniformly continuous and satisfies the Lipschitz condition in (\ref{C3r}). Then, the length of the confidence bands and intervals for $\rho$ have the orders of magnitude $n^{-\beta /(2 \beta+2)}$ and $n^{-\beta /(2 \beta+1)}$, respectively. This follows from Corollary \ref{co32} and (\ref{z44}).  In particular, if $\beta=2$, such orders become $n^{-1/3}$ and $n^{-2/5}$, respectively.

The MSE of kernel density estimators are usually evaluated under the assumption that $\rho$ has at least two continuous derivatives (see, for instance Tsybakov \cite{tsyba} or Scott \cite{scott}). In the following result, we compute the MSE of $L_{2, h}^{(1)}(\hat{F}_{n} ; x)$ only assuming that $\rho\in {\cal C}_{S}(\mathbb{R})$.

 \begin{thm} Let $F \in  {\cal C}_{S}^{1}(\mathbb{R})$ and let $\widetilde{R}_{1}$ be an independent copy of $\widetilde{S}_{1}$. Then,
	\begin{align}
		&MSE\left(L_{2, h}^{(1)}\left(\hat{F}_{n} ; x\right)\right) \nonumber \\
		&=\frac{1}{4 n h^{2}}\left\{\widetilde{\mathbb{E}}\left(F\left(x+h\left(\widetilde{R}_{1} \wedge \widetilde{S}_{1}\right)+h\right)-F\left(x+h\left(\widetilde{R}_{1} \vee \widetilde{S}_{1}\right)-h\right)\right)\right.\nonumber\\ \phantom{\frac{1}{4 n h^{2}}}&\left. -\left(\widetilde{\mathbb{E}}\left(F\left(x+h \widetilde{S}_{1}+h\right)-F\left(x+h \widetilde{S}_{1} - h\right)\right)\right)^{2}\right\} +\left(L_{2, h}(\rho ; x)-\rho_{\star}(x)\right)^{2}  \nonumber\\
		&\leq \frac{\rho_{\star}(x)}{2 n h}+\frac{1}{2 n h} \left(L_{2, h}(\rho ; x)-\rho_{\star}(x)\right)+\left(L_{2, h}(\rho ; x)-\rho_{\star}(x)\right)^{2}. \label{67a}
	\end{align} \label{pra4}
\end{thm}

If $\rho\in{\cal C}_{S}(\mathbb{R})$, Cline and Hart \cite{clhake} obtained asymptotic expansions for the MSE  of kernel estimators of order $m$, also providing a test for the null hypothesis that $\rho$ is continuous at $x_0$ (in this respect, see also Funke and Hirukawa \cite{fuhike}).

If $\rho_{\star}(x)=0$, the upper bound in (\ref{67a}) is very small and therefore the corresponding confidence interval is asymptotically shorter than that in Theorem \ref{te45}, as shown in the following result. To this end, denote
\begin{equation}
	g_{\beta}(h)=\frac{h^{\beta-1}}{2 n}+ A(x) h^{2 \beta},\quad 0<\beta \leq 2, \quad 0<h\leq H(x),\label{n4}
\end{equation}
where $A(x)$ and $H(x)$ are the same as in (\ref{z5}).
The next result provides confidence intervals in the case  $\rho_{\star}(x)=0$, which entails that $\rho(x)=0$.
\begin{thm} Assume that $\rho \in {\cal C}_{S}(\mathbb{R})$, $\rho_{\star}(x)=0$, and $H(x)>0$. If the function $\rho_{x}$ satisfies (\ref{z5}), then
	\begin{enumerate}[(a)]
\item		If $\beta\geq 1$, then we have for any $0<h\leq H(x)$
\begin{equation}
	P\left(L_{2, h}^{(1)}\left(\hat{F}_{n} ; x\right) \geq \frac{1}{\sqrt{\alpha}} \sqrt{A(x) g_{\beta}(h)}\right) \leq \alpha . \label{o4}
\end{equation}
\item If $0<\beta<1$, then
	\begin{equation}
	P\left(L_{2, h_\beta}^{(1)}\left(\hat{F}_{n} ; x\right) \geq \frac{1}{\sqrt{\alpha}} \sqrt{A(x) g_{\beta}\left(h_{\beta}\right)}\right) \leq\alpha , \label{p4}
\end{equation}
where
\begin{equation}
	h_{\beta}=\left(\frac{1-\beta}{4 \beta A(x)n}\right)^{1 /(\beta+1)}, \label{q4}
\end{equation}
provided that $h_{\beta}\leq H(x)$.
\end{enumerate}\label{te47}
\end{thm}
From (\ref{n4}), we have the following. If $1<\beta \leq 2$, the length of the confidence interval in (\ref{o4}) is as short as we wish, whenever we choose an arbitrary small $h>0$. If $\beta=1$, such a length is of order $1 / \sqrt{n}$. On the other hand, if $0<\beta<1$, it follows from $(\ref{n4})$ and (\ref{q4}) that the length of the confidence interval in $(\ref{p4})$ is of order $n^{-\beta /(\beta+1)}$.

 Let $k \in \mathbb{N}$. We finally consider the problem of giving confidence intervals for the $k$th derivative of $\rho$. Rewrite (\ref{31v3}) in terms of $\rho$ as
\begin{equation}
	L_{k+2, h}^{(k+1)}\left(\hat{F}_{n} ; x\right)-(\rho^{(k)})_{\star}(x)=T_{k+1, n}+L_{k+2, h}\left(\rho^{(k)} ; x\right)-(\rho^{(k)})_{\star}(x) . \label{r4}
\end{equation}
Denote
\begin{equation}
	d_{k+1}=\binom{k}{\lfloor(k+1) / 2\rfloor}, \quad c_{k+1}=\frac{1}{d_{k+1}^{2}}\binom{2 k}{k} 2 h(\rho(x)+r), \quad r \geq 0, \label{s4}
\end{equation}
where $\lfloor y\rfloor$ stands for the integer part of $y\in \mathbb{R}$.
\begin{thm} Let $k \in \mathbb{N}$ and $r, s \geq 0$. Assume that $\rho \in {\cal C}_{S}^{k}(\mathbb{R})$ and that $\rho(x)>0$.  Choose the bandwidth $h:=h(x, n)$ in such a way that
\begin{equation}
	\left|D_{k+2, k, h}(\rho ; x)-\rho(x)\right| \leq r, \quad\left|L_{k+2, h}\left(\rho^{(k)} ; x\right) - (\rho^{(k)})_{\star}(x)\right| \leq \frac{1}{(2 h)^{k}} \sqrt{\frac{s}{2 n h}} . \label{t4}
\end{equation}
Then,
\begin{equation}
	P\left(\left|L_{k+2, h}^{(k+1)}\left(\hat{F}_{n} ; x\right) -(\rho^{(k)})_{\star}(x)\right| \geq l_{k+1}(x,n)\right) \leq \alpha, \label{u4}
\end{equation}
where
\begin{equation}
	l_{k+1}(x,n)=c_{k+1} d_{k+1} \tau^{-1}\left(\frac{\log (2 / \alpha)}{n c_{k+1}}\right)\frac{1}{(2 h)^{k+1}}+\frac{1}{(2 h)^{k}} \sqrt{\frac{s}{2 n h}} . \label{v4}
\end{equation} \label{te48}
\end{thm}
\begin{rem}
	Proceeding as in Remark \ref{rem42} and taking into account (\ref{s4}), it can be checked that
\begin{equation}
	l_{k+1}(x,n) \sim \frac{1}{(2 h)^{k}} \frac{1}{\sqrt{2 n h}}\left(t_{\alpha} \sqrt{\binom{2 k}{k}(\rho(x)+r)}+\sqrt{s}\right), \quad n h^{2 k+1} \rightarrow \infty . \label{w4}
\end{equation}
\end{rem}
Similar statements to those made after Corollary \ref{cor99} apply when $x$ is close to a discontinuity point of $\rho^{(k)}$.

 Let $x \in[a, b]$.  Suppose that $\rho$ is  a polynomial of degree $k+1$ at most on  $[a-(k+2)h_0, b+(k+2)h_0]$, where $h_0>0$ satisfies the first inequality in (\ref{t4}). Since the  operator $L_{k+2, h_0}$ preserves affine functions, the second inequality in (\ref{t4}) is fulfilled for $s=0$. Therefore, choosing $h=h_0$ in Theorem \ref{te48}, we have from (\ref{w4})
\[
l_{k+1}(x,n) \sim \frac{1}{\left(2 h_{0}\right)^{k}} \frac{1}{\sqrt{2 n h_{0}}} t_{\alpha} \sqrt{\binom{2 k}{k}(\rho(x)+r)}, \quad n \rightarrow \infty .
\]
Hence, this length is of order $1 / \sqrt{n}$.

On the other hand, suppose that the function $\left(\rho^{(k)}\right)_{x}$ satisfies the Lipschitz condition
\begin{equation*}
	\omega_{2}\left(\left(\rho^{(k)}\right)_{x}, [x-(k+2)h, x+(k+2)h]; h\right) \leq A(x) h^{\beta},  \quad  0 <\beta \leq 2, \label{C3}
\end{equation*}
for some positive constant $A(x)$.  By Proposition \ref{prbost} and (\ref{B3}), the second inequality in (\ref{t4}) is satisfied if
\begin{equation}
	e_{k+2} A(x) h^{\beta} \leq \frac{1}{(2 h)^{k}} \sqrt{\frac{s}{2 n h}} . \label{x4}
\end{equation}
Note that the first inequality in (\ref{t4}) is always fulfilled for a small enough $h$. Dismissing the constants in (\ref{x4}), we can choose $h\sim n^{-1/(2 \beta+2 k+1)}$. In such a case, the order of magnitude of the length $l_{k+1}(x,n)$ in (\ref{w4}) is
\begin{equation}
	l_{k+1}(x,n) \sim n^{-\beta /(2 \beta+2 k+1)} . \label{y4}
\end{equation}

Suppose that $\rho^{(k)}$ is uniformly continuous and satisfies the Lipschitz condition in (\ref{C3r}). In such a case, we have seen in Corollary \ref{co32} that the length $l_{k+1}$ of the confidence band for $\rho^{(k)}$ is of order $l_{k+1} \sim n^{-\beta /(2 \beta+2 k+2)}$, which is worse than that in (\ref{y4}).

Finally, we give the following upper bound for the MSE of the estimator $L_{k+2, h}^{(k+1)}(\hat{F}_{n};x)$.
\begin{thm} Let $k \in \mathbb{N}_{0}$ and assume that $F^{(1)}=\rho \in {\cal C}_{S}^{k}(\mathbb{R})$. Then,
	\begin{equation*}
		MSE\left(L_{k+2, h}^{(k+1)}\left(\hat{F}_{n} ; x\right)\right) \leq\binom{ 2 k}{k} \frac{1}{n(2 h)^{2 k+1}} D_{k+2, k, h}(\rho ; x)+\left(L_{k+2, h}\left(\rho^{(k)} ; x\right)-(\rho^{(k)})_{\star}(x)\right)^{2}.
	\end{equation*} \label{te413}
\end{thm}

Note that the upper bound in Theorem \ref{te413} is accurate.  In fact, for $k=0$ this upper bound is the same as that in (\ref{67a}).

\section{Simulations} \label{sesimu}

In this section, we illustrate Corollaries \ref{co44} and \ref{cor99}. In the first place, we consider a random variable $X$ having the Poisson distribution with mean $4$ and associated midpoint distribution function $F_{\star}(x)$. Let $m \in \mathbb{N}_{0}$. According to Corollary \ref{co44}, the bandwidth $h$ is selected as follows: if $x=m$, then $h=1$, whereas if $x \in(m, m+1)$, then $h=\min (x - m,(m+1)-x)$. For $\alpha=0.05$, we have
\begin{equation}
	l_{0}(x, n) \sim \frac{2.716}{\sqrt{n}}(\sigma \circ F)_{\star}(x), \quad n \rightarrow \infty, \label{poisi}
\end{equation}
as follows from Remark \ref{rem42}. Thus, the maximum asymptotic length in (\ref{poisi}) is attained when $x$ is near the median of $X$. 

In Figure \ref{fig:grafico1}, we plot $95\%$ confidence intervals for $F_{\star}(x)$ at the discontinuity points and at three intermediate points between discontinuities, choosing a sample size $n=200$.

\vspace{-1cm}

\begin{figure}[htbp]
    \centering
      \includegraphics[width=0.9\textwidth]{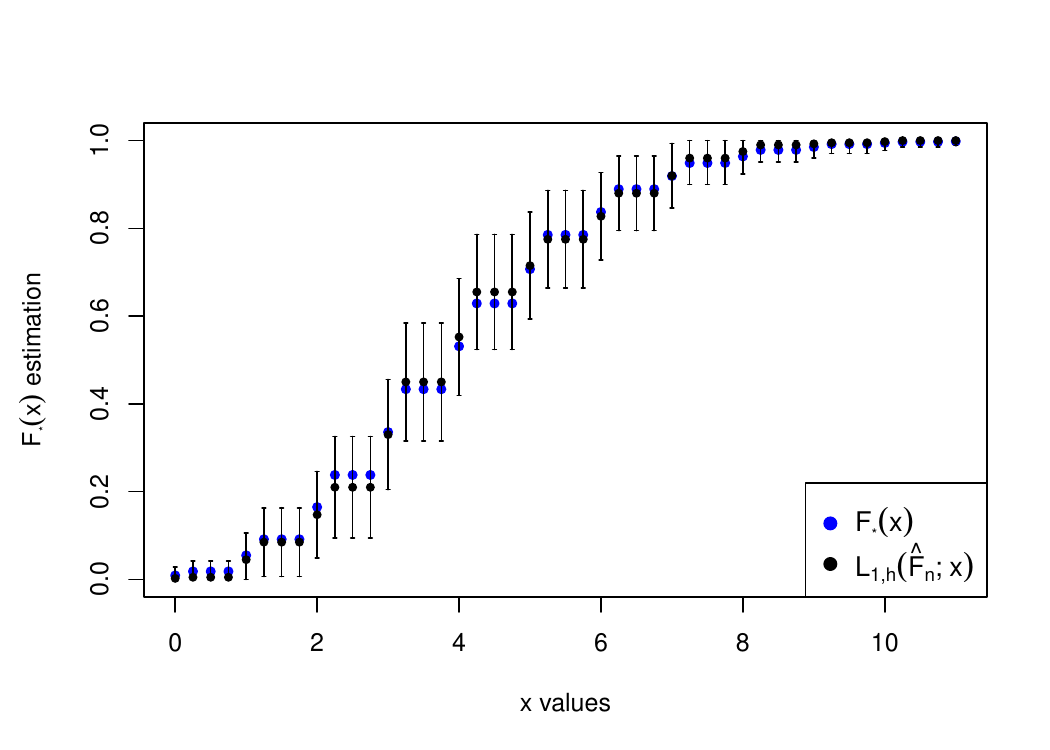}
        \caption{ 95$\%$ confidence intervals and estimates (black points and error bars) for $F_{\star}(x)$ (solid blue points), using a sample size $n=200$.}
        \label{fig:grafico1}
\end{figure}

On the other hand, let $a>0$ be fixed. To illustrate Corollary \ref{cor99}, we consider in the rest of this section the Pareto density
\begin{equation}
	\rho(x)=\frac{a}{(x+1)^{a+1}} 1_{(0, \infty)}(x)+\frac{a}{2} 1_{\{0\}}(x), \quad  x \in \mathbb{R} . \label{80p}
\end{equation}
Observe that $\rho_{\star} =\rho$. The following result gives confidence intervals for $\rho(0)=a / 2$.
\begin{prop} We have
\[		P\left(\left|L_{2, h_{0}}^{(1)}\left(\hat{F}_{n} ; 0\right)-\rho(0)\right| \geq l_{1}(0, n)\right) \leq \alpha ,
\]
where
\[l_{1}(0, n) \sim \frac{3 t_{\alpha} \sqrt{a}}{4} \frac{1}{\sqrt{n h_{0}}}, \quad h_{0}=\left(\frac{9 t_{\alpha}^{2}}{16 a (a+1)^{2} n}\right)^{1 / 3} . \]
\label{pr61}
\end{prop}

\begin{proof} Using a Taylor expansion of first order around the origin, we get from (\ref{80p})\\
\[\left|L_{2, h}(\rho ; 0)-\rho(0)\right|=\left|\widetilde{\mathbb{E}} \rho\left(h \widetilde{S}_{2}\right) 1_{\left\{\widetilde{S}_{2}>0\right\}}-\frac{a}{2}\right|\leq h\left\|\rho^{(1)}\right\|_{[0,\infty)} \widetilde{\mathbb{E}} \widetilde{S}_{2} 1_{\left\{ \widetilde{S}_{2}>0\right\} }\leq\frac{a(a+1)}{3} h.\]
Hence, the result follows from Corollary \ref{cor99} by setting $x=0, \ H(0)=\infty$, $A(0)=a(a+1) / 3$, and $\beta=1$.
\end{proof}

\begin{prop} Let $x>0$. Then,
\begin{equation}
	P\left(\left|L_{2, h_{0}}^{(1)}\left(\hat{F}_{n} ; x\right)-\rho(x)\right| \geq l_{1}(x, n)\right) \leq \alpha, \label{81p}
\end{equation}
where
\[
l_{1}(x, n) \sim \frac{5 t_{\alpha}}{4} \sqrt{\frac{a}{(x+1)^{a+1}}} \frac{1}{\sqrt{2 n h_{0}}}, \quad h_{0}=\left(\frac{9 t_{\alpha}^{2}}{32 a(a+1)^{2}(a+2)^{2}} \frac{(x\vee 1)^{2 a+6}}{(x+1)^{a+1}} \frac{1}{n}\right)^{1 / 5}, \]
 whenever $h_{0} \leq(x\wedge 1) / 2$.\label{pr62}
\end{prop}

\begin{proof} It is easily checked that
\[
\left\|\rho^{(2)}\right\|_{[x-2 h, x+2 h]} \leq \frac{a(a+1)(a+2)}{(x\vee 1)^{a+3}}=: B(x), \quad h \leq(x\wedge 1) / 2 .
\]
Therefore, using a Taylor's expansion of second order around $x$, we have

\[\left|L_{2, h}(\rho ; x)-\rho(x)\right|=\left|\widetilde{\mathbb{E}} \rho\left(x+h \widetilde{S}_{2}\right)-\rho(x)\right| \\
\leq \frac{h^{2}}{2}\left\|\rho^{(2)}\right\|_{[x-2 h, x+2 h]} \widetilde{\mathbb{E}} \widetilde{S}_{2}^{2} \leq \frac{B(x)}{3} h^{2} .\]
Thus, the result follows from Corollary \ref{cor99} by choosing $H(x)=(x \wedge 1) / 2, \  A(x)=B(x) / 3$, and $\beta=2$. 	
\end{proof}

Let $0 < x\leq 1$. A sufficient condition for $h_{0} \leq(x \wedge 1) / 2$ is that
\begin{equation}
	\frac{9 t_{\alpha}^{2}}{a(a+1)^{2}(a+2)^{2} n} \leq x^{5} . \label{82p}
\end{equation}
Comparing Propositions \ref{pr61} and \ref{pr62}, we see that $l_{1}(0, n) \sim n^{-1 / 3}$, whereas $l_{1}(x, n) \sim n^{-2 / 5}, \ x>0$. However, if $x>0$ is close to the origin, we need a very large sample size $n$ to guarantee (\ref{81p}), as follows from (\ref{82p}).
Figure \ref{fig:grafico2} below gives us $95\%$ confidence intervals for the Pareto density $\rho=\rho_{\star}$ defined in (\ref{80p}) with $a=1$, when the sample size is $n=1500$. It should be pointed out that such confidence intervals are conservative. Also, observe that once $n=1500$ is chosen, confidence intervals are possible for $x\geq x_0=0.262$, where $x_0$ satisfies (\ref{82p}). 

\vspace{-1cm}
\begin{figure}[h]
	\centering
	\includegraphics[width=0.9\textwidth]{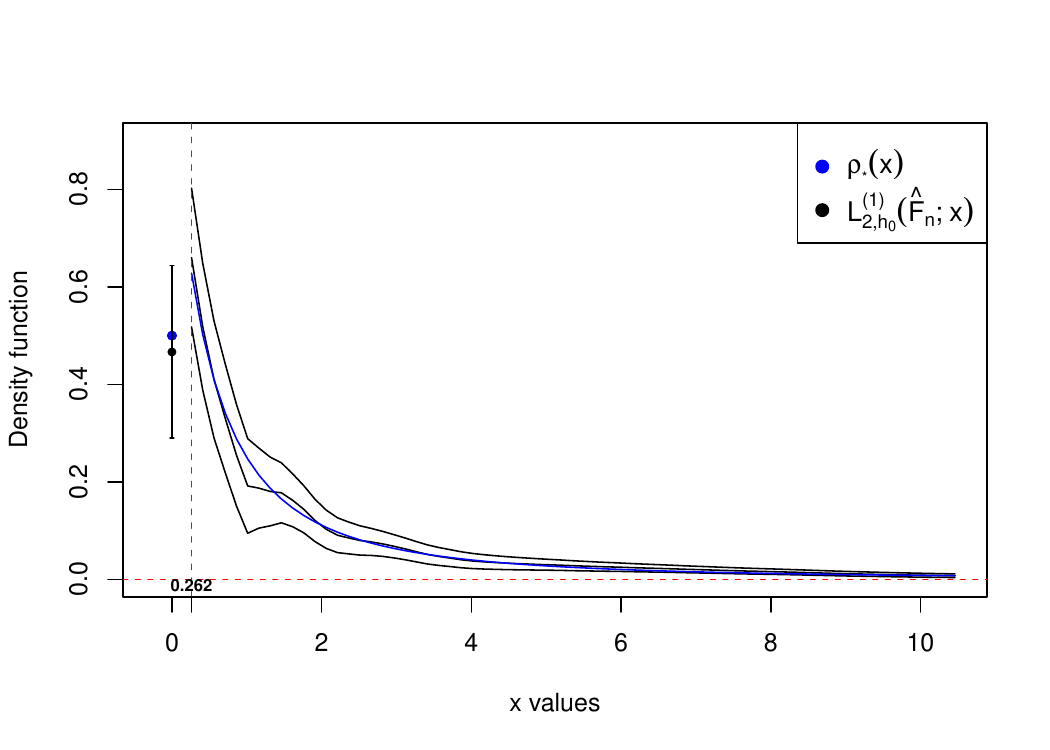}
	\caption{ 95$\%$ confidence intervals at $x=0$ and $x>x_0=0.262$, where $x_0$ satisfies \eqref{82p}, for the Pareto density (solid blue line) with $a=1$ and sample size $n=1500$. }
	\label{fig:grafico2}
\end{figure}
	\appendix
\section{Appendix} \label{seappe}
Let $\mathbb{Z}=(Z(\theta), \ -1 \leq \theta \leq 1)$ be a stochastic process such that
\begin{equation}
	\mathbb{E} Z(\theta)=0, \quad|Z(\theta)| \leq d, \quad-1 \leq \theta \leq 1, \label{44a}
\end{equation}
for some $d>0$. Let $\mathbb{Z}_{(j)}, j=1, \dots, n$, be a sequence of independent copies of $\mathbb{Z}$ and define
\begin{equation}
	W_{n}(\theta)=Z_{(1)}(\theta)+\dots +Z_{(n)}(\theta), \quad-1 \leq \theta \leq 1. \label{45a}
\end{equation}
We will use the following continuous form of Hölder's inequality shown by Kwon \cite{kwexte}. If $f(\omega, \widetilde{\omega})$ is a positive measurable function defined on $\Omega \times \widetilde{\Omega}$, then
\begin{equation}\int_{\Omega} \exp \left(\int_{\widetilde{\Omega}} \log f(\omega, \widetilde{\omega}) d \widetilde{P}(\widetilde{\omega})\right) d P(\omega) \leq \exp \left(\int_{\widetilde{\Omega}} \log \left(\int_{\Omega} f(\omega, \widetilde{\omega}) d P(\omega)\right) d \widetilde{P}(\widetilde{\omega})\right) \label{47a}\end{equation}

The following concentration inequality will be applied to estimate the tail probabilities of the random term defined in (\ref{dftkh}). A similar inequality can be found in Babu et al. \cite[Lemma 2.1]{baappl}.

\begin{lem}Let $\mathbb{Z}$ be as in (\ref{44a}) and let $\widetilde{R}$ be a $[-1,1]$-valued random variable. For any $c \geq d^{-2} \mathbb{E}\times\widetilde{\mathbb{E}} Z^{2}(\widetilde{R})$, we have
\[P\left(\left|\widetilde{E} W_n(\widetilde{R})\right| \geq n c d \delta\right) \leq 2 e^{- n c \tau(\delta)}, \quad \delta> 0,\] where the function $\tau(\delta)$ is defined in (\ref{a4}).\label{lea1} \end{lem}
\begin{proof} Applying (\ref{47a}) to the function
\[
f(w, \widetilde{w})=\exp \left(t W_{n}(\widetilde{R}(\widetilde{w}))(w)\right),\quad t \in \mathbb{R},
\]
we obtain
\begin{equation}
	\mathbb{E} e^{t \widetilde{\mathbb{E}} W_{n}(\widetilde{R})} \leq \exp \left(\widetilde{\mathbb{E}} \log \mathbb{E} e^{t W_{n}(\widetilde{R})}\right), \quad t \in \mathbb{R} . \label{48a}
\end{equation}
Let $\beta>0$. We claim that
\begin{equation}
	\max \left(\mathbb{E} e^{\beta \widetilde{\mathbb{E}} W_{n}(\widetilde{R})}, \mathbb{E} e^{-\beta \widetilde{\mathbb{E}} W_{n}(\widetilde{R})}\right) \leq e^{n c\left(e^{\beta d} - 1 - \beta d\right)} . \label{49a}
\end{equation}
In fact, it was show in \cite[lemma A.1]{adalra} that
\[
\max \left(\mathbb{E} e^{\beta W_{n}(\theta)}, \mathbb{E} e^{-\beta W_{n}(\theta)}\right) \leq \exp \left(n \frac{e^{\beta d} - 1-\beta d}{d^{2}} \mathbb{E} Z^{2}(\theta)\right), \ 1 \leq \theta \leq 1 .
\]
We therefore have from (\ref{48a}) and Fubini's theorem
\begin{align*}
&\mathbb{E} e^{\beta \widetilde{\mathbb{E}} W_{n}(\widetilde{R})} \leq \exp \left(\widetilde{\mathbb{E}} \log \mathbb{E} e^{\beta W_{n}(\widetilde{R})}\right) \\&\leq \exp \left(n \frac{e^{\beta d} - 1 - \beta d}{d^{2}} \mathbb{E}\times\widetilde{\mathbb{E}} Z^{2}(\widetilde{R})\right) \leq e^{n c\left(e^{\beta d} - 1 -\beta d\right)}.
\end{align*}
The same inequality holds replacing $\beta$ by $-\beta$. This shows claim (\ref{49a}).

Let $\delta>0$.  From (\ref{49a}), we get
\begin{align}
&P\left(\left|\widetilde{\mathbb{E}} W_{n}(\widetilde{R})\right| \geq n c d \delta\right) \leq e^{-\beta n c d \delta} \mathbb{E} e^{\beta\left|\widetilde{\mathbb{E}} W_{n}(\widetilde{R})\right|} \nonumber\\ & \leq 2 \exp \left(n c\left(e^{\beta d}-1 - \beta d(1+\delta)\right)\right) . \label{50a}
\end{align}
Choosing the value of $\beta$ minimizing the exponent in (\ref{50a}), that is, $e^{\beta d} - 1=\delta$, the result follows from \rf{a4}.
\end{proof}
Let $k \in N_{0}$. Let  $T_{k, n}$ be the random variable defined in (\ref{dftkh}), that is,
\begin{equation}
	T_{k, n}=\frac{1}{n(2 h)^{k}} \widetilde{\mathbb{E}} \Delta_{h}^{k}\left(S_{n}\left(F\left(x+h \widetilde{S}_{1}\right)\right)- nF\left(x+h \widetilde{S}_{1}\right)\right). \label{51a}
\end{equation} In view of (\ref{51a}), we define the centered stochastic process $\mathbb{Z}_{k}$ by
\begin{equation}
	Z_{k}(\theta)=\Delta_{h}^{k}\left(S_{1}(F(x+h \theta))-F(x+h \theta)\right), \quad-1 \leq \theta \leq 1 . \label{52a}
\end{equation}
Note that the dependence on $x$ and $h$ in (\ref{51a}) and (\ref{52a}) is dropped. Let $\mathbb{Z}_{k, j}$, $j= 1, \dots, n$, be a sequence of independent copies of $\mathbb{Z}_{k}$ and define
\begin{equation}
	W_{k, n}(\theta)=Z_{k, 1}(\theta)+\dots +Z_{k, n}(\theta), \quad-1 \leq \theta \leq 1, \label{53a}
\end{equation}
so that we can write
\begin{equation}
	T_{k, n}=\frac{1}{n(2 h)^{k}} \widetilde{\mathbb{E}} W_{k, n}\left(\widetilde{S}_{1}\right). \label{54a}
\end{equation}
To estimate the tail probabilities of $T_{k,n}$ by means of Lemma \ref{lea1}, the following auxiliary result will be needed.

\begin{lem} Let $k \in \mathbb{N}_{0}$. Then.
\begin{equation}
	\mathbb{E}\times\widetilde{\mathbb{E}} Z_{0}^{2}\left(\widetilde{S}_{1}\right)=L_{1, h}\left(\sigma^{2} \circ F ; x\right) . \label{55a}
\end{equation}
If, in addition, $F \in {\cal C}_{S}^{1}(\mathbb{R})$ with $F^{(1)}=\rho$, then
\begin{equation}
	\left|Z_{k+1}(\theta)\right|\leq\binom{k}{\lfloor(k+1) / 2\rfloor}=: d_{k+1}, \quad-1 \leq \theta \leq 1, \label{56a}
\end{equation}
and
\begin{equation}
	\mathbb{E}\times\widetilde{\mathbb{E}} Z_{k+1}^{2}\left(\widetilde{S}_{1}\right)\leq \binom{2k}{k}2hD_{k+2, k, h}(\rho ; x),\label{57a}
\end{equation}
where the operator $D_{k+2, k, h}$ is defined in (\ref{B}).
\end{lem}
\begin{proof}  Using (\ref{52a}) with $k=0$, we have from (\ref{dfstek}) and Fubini's theorem
\[
\mathbb{E}\times\widetilde{\mathbb{E}} Z_{0}^{2}\left(\widetilde{S}_{1}\right)=\widetilde{\mathbb{E}} F\left(x+h \widetilde{S}_{1}\right)\left(1-F\left(x+h \widetilde{S}_{1}\right)\right)=L_{1, h}\left(\sigma^{2} \circ F ; x\right),
\]
thus showing (\ref{55a}). On the other hand, let $-1 \leq \theta \leq 1$.  Denote
\begin{equation}
	x_{j}(\theta)=x+(2 j-k) h+h \theta, \quad j=0,1, \dots, k . \label{58a}
\end{equation}
Observe that
\begin{equation}
	x_{j-1}(\theta)+h=x_{j}(\theta)-h, \quad j=1, \dots, k . \label{59a}
\end{equation}
By (\ref{dfsydi}), (\ref{dfdire}), and (\ref{58a}), we can rewrite (\ref{52a}) as
\begin{align}
&Z_{k+1}(\theta)=\Delta_{h}^{k}\left(\Delta_{h}^{1}\left(S_{1}\left(F(x+h\theta)\right)-\Delta_{h}^{1}F(x+h\theta) \right) \right)\nonumber\\
&=\sum_{j=0}^{k}\binom{k}{j}(-1)^{k-j}\left(\Delta_{h}^{1}\left(S_{1}\left( F\left(x_{j}(\theta)\right)\right)-\Delta_{h}^{1}F\left(x_{j}(\theta)\right)\right)\right)\nonumber\\
	&=\sum_{j=0}^{k}\binom{k}{j}(-1)^{k-j}\left(S_{1}\left(F\left(x_{j}(\theta)+h\right)\right)-S_{1}\left(F\left(x_{j}(\theta)-h\right)\right)\right)-\Delta_{h}^{k+1} F(x+h \theta)\nonumber \\
	&=: G_{k+1}(\theta)-\Delta_{h}^{k+1} F(x+h \theta). \label{60a}
\end{align}
From (\ref{59a}) and the second equality in (\ref{60a}), we see that
\[
\left|Z_{k +1}(\theta)\right| \leq \max _{0 \leq j \leq k}\binom{k}{j}=\binom{k}{\lfloor(k+1) / 2\rfloor} .
\]
Finally, we have from (\ref{60a})
\begin{equation}
	\mathbb{E} Z_{k+1}^{2}(\theta)=\mathbb{E} G_{k+1}^{2}(\theta)-\left(\Delta_{h}^{k+1} F(x+h \theta)\right)^{2} \leq \mathbb{E} G_{k+1}^{2}(\theta) . \label{61a}
\end{equation}
By (\ref{59a}), the probability law of $G_{k+1}(\theta)$ is given by
\begin{equation}
	P\left(G_{k+1}(\theta)=\binom{k}{j}(-1)^{k - j}\right)=F\left(x_{j}(\theta)+h\right)-F\left(x_{j}(\theta) - h\right), \quad j=1, \dots, k, \label{62a}
\end{equation}
and
\[
P\left(G_{k+1}(\theta)=0\right)=F\left(x_{0}(\theta) - h\right)+1-F\left(x_{k}(\theta)+h\right) .
\]
Let $\widetilde{V}$ be a random variable uniformly distributed on $[-1,1]$, independent of $\widetilde{S}_{1}$ and $\widetilde{T}_{k}$, as defined in (\ref{A}). From (\ref{58a}) and (\ref{62a}), we get
\begin{align*}
	&\mathbb{E} G^{2}_{k+1}(\theta)=\sum_{j=0}^{k}\binom{k}{j}^{2}\left(F\left(x_{j}(\theta)+h\right)-F\left(x_{j}(\theta)-h\right)\right)=2 h \sum_{j=0}^{k}\binom{k}{j}^{2} \widetilde{\mathbb{E}} \rho\left(x_{j}(\theta)+h \widetilde{V}\right)\\
&=2 h\binom{2 k}{k} \sum_{j=0}^{k} \frac{\binom{k}{j}^{2}}{\binom{2 k}{k}} \widetilde{\mathbb{E}} \rho(x+(2 j-k) h+h \theta+h \widetilde{V})\\
&=2 h\binom{2 k}{k} \widetilde{\mathbb{E}} \rho\left(x+h\left(\widetilde{T}_{k}+\theta+\widetilde{V}\right)\right).
\end{align*}
Hence, we have from (\ref{61a}) and Fubini's theorem
\[\mathbb{E} \times \widetilde{\mathbb{E}} Z_{k+1}^{2}\left(\widetilde{S}_{1}\right) \leq 2 h\binom{2 k}{k} \widetilde{\mathbb{E}} \rho\left(x+h\left(\widetilde{T}_{k}+\widetilde{S}_{2}\right)\right)=2 h\binom{2 k}{k} D_{k+2, k, h}(\rho ; x),\] where the last equality follows from (\ref{B}). This completes the proof.
\end{proof}

\begin{proof}[\bf Proof of Theorem \ref{te41}] By (\ref{b4}) and (\ref{55a}), we have
	\begin{equation}
		\mathbb{E} \times \widetilde{\mathbb{E}} Z_{0}^{2}\left(\widetilde{S}_{1}\right) \leq\left(\sigma^{2} \circ F\right)_{\star}(x)+\nu_{n}=: c_{0} . \label{72a}
	\end{equation}
	Let $\delta>0$. From (\ref{b4}) and identity (\ref{63a}) below, we see that
	\begin{align}
		&P\left(\left|L_{1, h}\left(\hat{F}_{n} ; x\right)-F_{\star}(x)\right| \geq c_{0} \delta+\sqrt{\frac{\epsilon_{n}}{n}}\right)\nonumber \\ &\leq P\left(\left|T_{0, n}\right| \geq c_{0} \delta\right)=P\left(\left| \widetilde{\mathbb{E}} W_{0, n}\left(\widetilde{S_{1}} \right)\right| \geq n c_{0} \delta\right), \label{73a}
	\end{align}
	where the last equality follows from (\ref{54a}). We apply Lemma \ref{lea1} with $d=1$, replacing $W_{n}, \widetilde{R}$ and $c$ by $W_{0, n}, \widetilde{S}_{1}$, and $c_{0}$, respectively, to obtain
	\begin{equation}
		P\left(\left|\widetilde{\mathbb{E}} W_{0, n}\left(\widetilde{S}_{1}\right)\right| \geq n c_{0} \delta\right) \leq 2 e^{-n c_{0} \tau(\delta)} . \label{74a}
	\end{equation}
	The result follows from (\ref{72a})-(\ref{74a}) by choosing $\delta$ in such a way that
	\[
	\tau(\delta)=\frac{\log (2 / \alpha)}{n c_{0}} .
	\]
\end{proof}

\begin{proof}[\bf Proof of Theorem \ref{pra3}]  Recalling (\ref{31v3}), we start with the identity
	\begin{equation}
		L_{1, h}\left(\hat{F}_{n} ; x\right)-F_{\star}(x)=T_{0, n}+L_{1, h}(F ; x)-F_{\star}(x),\quad  F \in {\cal C}_{S}(\mathbb{R}) , \label{63a}
	\end{equation}
	thus implying that
\begin{equation}
	MSE\left(L_{1, h}\left(\hat{F}_{n} ; x\right)\right)=\mathbb{E} T_{0, n}^{2}+\left(L_{1, h}(F ; x)-F_{\star}(x)\right)^{2} . \label{65a}
\end{equation}
Let $u, v \in[0,1]$. It is easily checked that
\[
\mathbb{E}\left(S_{1}(u)-u\right)\left(S_{1}(v) - v\right)=(u \wedge v)(1-(u \vee v)) .
\]
Thus, the increase of $F$ and (\ref{52a}) imply that
\[
\mathbb{E} Z_{0}(u) Z_{0}(v)=F(x+h(u \wedge v))(1-F(x+h(u \vee v)) .
\]
In turn, this and (\ref{53a}) entail that
\[\mathbb{E}W_{0, n}(u) W_{0, n}(v) =n\mathbb{E}Z_{0}(u) Z_{0}(v)=n F(x+h(u \wedge v))(1-F(x+h(u \vee v)).\]
We therefore have from (\ref{54a}) and Fubini's theorem
\begin{align*}&\mathbb{E} T_{0, n}^{2}=\frac{1}{n^{2}} \mathbb{E}\left(\widetilde{\mathbb{E}} W_{0, n}\left(\widetilde{S}_{1}\right)\right)^{2}=\frac{1}{n^{2}} \mathbb{E}\times\widetilde{\mathbb{E}} W_{0, n}\left(\widetilde{R}_{1}\right) W_{0, n}\left(\widetilde{S}_{1}\right)\\
&=\frac{1}{n} \widetilde{\mathbb{E}} F\left(x+h\left(\widetilde{R}_{1} \wedge \widetilde{S}_{1}\right)\right)\left(1-F\left(x+h\left(\widetilde{R}_{1} \vee \widetilde{S}_{1}\right)\right)\right).\end{align*}
Hence, the equality in (\ref{64a}) follows from (\ref{65a}). The inequality in (\ref{64a}) is an immediate consequence of the increase of $F$. \end{proof}

\begin{proof}[\bf Proof of Theorem \ref{te43}]  By Theorem \ref{pra3} and (\ref{b4}), we see that
\begin{align*}
MSE\left(L_{1, h}\left(\hat{F}_{n} ; x\right)\right) & \leq \frac{1}{n}\left|L_{1, h}\left(\sigma^{2} \circ F ; x\right)-\left(\sigma^{2} \circ F\right)_{\star}(x)\right|+\left(L_{1, h}(F ; x)-F_{\star}(x)\right)^{2} \\
	& \leq \frac{\epsilon_{n}+\nu_{n}}{n} .
\end{align*}
Hence, the result follows from Chebyshev's inequality.
\end{proof}

\begin{proof}[\bf Proof of Theorem \ref{te45}] As follows from (\ref{C}), $D_{2,0, h}=L_{2, h}$. Thus, we have from (\ref{h4}) and (\ref{57a})
\begin{equation}
	\mathbb{E}\times\widetilde{\mathbb{E}} Z_{1}^{2}\left(\widetilde{S}_{1}\right) \leq 2 h L_{2, h}(\rho ; x) \leq 2 h\left(\rho_{\star}(x)+\sqrt{\frac{r}{2 n h}}\right)=: c_{1} . \label{75a}
\end{equation}
Let $\delta>0$. From (\ref{h4}) and identity (\ref{66a}) below, we see that
\begin{align}
&	P\left(\left|L_{2, h}^{(1)}\left(\hat{F}_{n} ; x\right)-\rho_{\star}(x)\right| \geq c_{1} \delta+\sqrt{\frac{r}{2 n h}}\right) \leq P\left(\left|T_{1, n}\right| \geq c_{1} \delta\right) \nonumber  \\
	&=P\left(\left|\widetilde{\mathbb{E}} W_{1, n}\left(\widetilde{S}_{1}\right)\right| \geq n c_{1} 2 h \delta\right),  \label{76a}
\end{align}
where the last equality follows from (\ref{54a}).

By (\ref{56a}), $d_{1}=1$. Therefore, applying Lemma \ref{lea1} with $d=d_{1}=1$, and replacing $W_{n}, \widetilde{R}$, and $c$, by $W_{1, n}, \widetilde{S}_{1}$, and $c_{1}$, respectively, we get
\begin{equation}
	P\left(\left| \widetilde{\mathbb{E}} W_{1, n}\left(\widetilde{S}_{1}\right) \right| \geq n c_{1} 2 h \delta\right) \leq 2 e^{-n c_{1} \tau(2 h \delta)} . \label{77a}
\end{equation}
The result follows from (\ref{75a})-(\ref{77a}) by choosing $\delta$ as
\[
\tau(2 h \delta)=\frac{\log (2 / \alpha)}{n c_{1}} .
\]
\end{proof}

\begin{proof}[\bf Proof of Corollary \ref{cor99}] Let $a>0$ and $0<t<1$. The function
\begin{equation}
f(r)=\frac{1}{r^{t/2}}(a+r^{1/2}),\quad r>0, \label{y1}
\end{equation}
attains its minimum at
\begin{equation}
r_0=\left(\frac{at}{1-t}\right)^2. \label{y2}
\end{equation}
In view of (\ref{h4}), we choose $r$ such that
\begin{equation}
A(x)h^{\beta}=\sqrt{\frac{r}{2nh}} \quad \Leftrightarrow \quad h=\left(\frac{r}{2A^2(x)n}\right)^{1/(2\beta+1)}=:Cr^{1/(2\beta+1)},\label{y3}
\end{equation}
so that $nh\rightarrow \infty$. We therefore have from Remark \ref{rem92} and (\ref{y3})
\begin{equation}
l_1(x,n)\sim \frac{1}{\sqrt{2nh}}\left(t_{\alpha}\sqrt{\rho_{\star}(x)}+\sqrt{r}\right)=\frac{1}{\sqrt{2nC}}\frac{1}{r^{1/2(2\beta+1)}}\left(t_{\alpha}\sqrt{\rho_{\star}(x)}+\sqrt{r}\right).\label{y41}
\end{equation}
Applying (\ref{y1}) and (\ref{y2}), the right-hand side in (\ref{y41}) attains its minimum at
\begin{equation}
r_0(x)=\frac{t_{\alpha}^2}{4\beta^2}\rho_{\star}(x).
\end{equation}
Thus, the result follows by replacing $r$ by $r_0(x)$ in (\ref{y3}) and (\ref{y41}).
\end{proof}

\begin{proof}[\bf Proof of Theorem \ref{pra4}]
	Assume that $F \in {\cal C}_{S}^{1}(\mathbb{R})$. By (\ref{31v3}), we have the identity
	\begin{equation}
		L_{2, h}^{(1)}\left(\hat{F}_{n} ; x\right)-\rho_{\star}(x)=T_{1, n}+L_{2, h}(\rho ; x)-\rho_{\star}(x) . \label{66a}
	\end{equation}As in the proof of Theorem $\ref{pra3}$, we have
	\begin{equation}
		MSE\left(L_{2, h}^{(1)}\left(\hat{F}_{n} ; x\right)\right)=\mathbb{E} T_{1,n}^{2}+\left(L_{2, h}(\rho ; x)-\rho_{\star}(x)\right)^{2} . \label{68a}
	\end{equation}
	Let $u, v \in[0,1]$. Using the increase of $F$ and the fact that $\mathbb{E} S_{1}(u) S_{1}(v)=u \wedge v$, it can be checked from (\ref{52a}) that
	\begin{align}
		&\mathbb{E} Z_{1}(u) Z_{1}(v)=F(x+h(u\wedge v)+h)-F(x+h(u \vee v)-h)\nonumber\\
		&	-(F(x+h u+h)-F(x+h u-h))(F(x+h v+h)-F(x+h v-h)) . \label{69a}
	\end{align}
	By (\ref{53a}), we have
	\[
	\mathbb{E} W_{1, n}(u) W_{1, n}(v)=n \mathbb{E} Z_{1}(u) Z_{1}(v) .
	\]
	Hence, we obtain from (\ref{54a}) and Fubini's theorem
	\begin{align*}
		&\mathbb{E} T_{1, n}^{2}=\frac{1}{(2 n h)^{2}} \mathbb{E}\left(\widetilde{\mathbb{E}} W_{1, n}\left(\widetilde{S}_{1}\right)\right)^{2}=\frac{1}{(2 n h)^{2}} \mathbb{E}\times\widetilde{\mathbb{E}} W_{1,n}\left(\widetilde{R}_{1}\right) W_{1, n}\left(\widetilde{S}_{1}\right)\nonumber\\
		&=\frac{1}{ 4 n h^{2}} \mathbb{E}\times\widetilde{\mathbb{E}} Z_{1}\left(\widetilde{R}_{1}\right) Z_{1}\left(\widetilde{S}_{1}\right)\\&=\frac{1}{4 n h^{2}}\left\{\widetilde{\mathbb{E}}\left(F\left(x+h\left(\widetilde{R}_{1} \wedge \widetilde{S}_{1}\right)+h\right)-F\left(x+h\left(\widetilde{R}_{1} \vee \widetilde{S}_{1}\right)-h\right)\right)\right.\nonumber\\
		&\left.-\left(\widetilde{\mathbb{E}}\left(F\left(x+h \widetilde{S}_{1}+h\right)-F\left(x+h \widetilde{S}_{1} - h\right)\right)\right)^{2}\right\}.
	\end{align*}
	This and (\ref{68a}) show the equality in (\ref{67a}).  Finally, observe that
	\begin{equation}
		\widetilde{\mathbb{E}}\left(F\left(x+h \widetilde{S_{1}}+h\right)-F\left(x+h \widetilde{S}_{1}-h\right)\right)=2 h \widetilde{\mathbb{E}} \rho\left(x+h \widetilde{S}_{2}\right)=2 h L_{2, h}(\rho ; x). \label{70a}
	\end{equation}
	The increase of $F$ implies that
	\begin{align*}&\widetilde{\mathbb{E}}\left(F\left(x+h\left(\widetilde{R}_{1} \wedge \widetilde{S}_{1}\right)+h\right)-F\left(x+h\left(\widetilde{R}_{1} \vee \widetilde{S}_{1}\right)-h\right)\right)\\ &\leq \widetilde{\mathbb{E}}\left(F\left(x+h \widetilde{S}_{1}+h\right)-F\left(x+h \widetilde{S}_{1} -h\right)\right).\end{align*}
	This, together with (\ref{70a}), shows the inequality in (\ref{67a}) and completes the proof.
\end{proof}

\begin{proof}[\bf Proof of Theorem \ref{te47}]  If $\rho_{\star}(x)=0$, we have from Proposition \ref{prbost}, Remark \ref{re14}, and (\ref{z5})
\[
L_{2, h}(\rho ; x) \leq \frac{7}{8} \omega_{2}(\rho_{ x},[x-2 h, x+2 h] ; h) \leq  A(x) h^{\beta} .
\]
This implies, by virtue of Theorem \ref{pra4} and (\ref{n4}), that
\begin{equation}
MSE\left(L_{2, h}^{(1)}\left(\hat{F}_{n} ; x\right)\right) \leq  A(x) \frac{h^{\beta}}{2 n h}+\left( A(x) h^{\beta}\right)^{2}= A(x) g_{\beta}(h). \label{78a}
\end{equation}
This and Chebyshev's inequality show part (a).

If $0<\beta<1$, we minimize the function $g_{\beta}(h)$. Straightforward computations show that $g_{\beta}(h)$ attains its minimum at $h=h_{\beta}$ as defined in (\ref{q4}). Therefore, part (b) follows from (\ref{78a}) and Chebyshev's inequality.\end{proof}

\begin{proof}[\bf Proof of Theorem \ref{te48}] By (\ref{s4}), (\ref{57a}), and the first inequality in (\ref{t4}), we have
\[
\frac{1}{d_{k+1}^{2}} \mathbb{E} \times \widetilde{\mathbb{E}} Z_{k+1}^{2}\left(\widetilde{S}_{1}\right) \leq \frac{1}{d_{k+1}^{2}}\binom{2 k}{k} 2 h(\rho(x)+r)=c_{k+1}.
\]

Let $\delta>0$. From (\ref{r4}) and the second inequality in \rf{t4}, we get
\begin{align}&
P\left(\left|L_{k+2, h}^{(k+1)}\left(\hat{F}_{n} ; x\right)-(\rho^{(k)})_{\star}(x)\right| \geq c_{k+1} d_{k+1} \delta+\frac{1}{(2 h)^{k}} \sqrt{\frac{s}{2 n h}}\right) \nonumber\\
&\leq P\left(\left|T_{k+1, n}\right| \geq c_{k+1} d_{k+1} \delta\right)=P\left(\left| \widetilde{\mathbb{E}} W_{k+1, n}\left(\widetilde{S_{1}} \right) \right|\geq n c_{k+1} d_{k+1}(2 h)^{k+1} \delta\right),\label{79a}
\end{align} where the last equality follows from (\ref{54a}). Applying Lemma \ref{lea1} with $d=d_{k+1}$, and replacing $W_n$, $\widetilde{R}$, and $c$, by $W_{k+1,n}$, $\widetilde{S}_1$, and $c_{k+1}$, we obtain
\begin{equation}
	P\left(\left|\widetilde{E} W_{k+1, n}\left(\widetilde{S}_{1}\right)\right| \geq n c_{k+1} d_{k+1}(2 h)^{k+1} \delta\right) \leq 2 e^{-n c_{k+1} \tau\left((2 h)^{k+1} \delta\right)} . \label{80a}
\end{equation}
The result follows from (\ref{79a}) and (\ref{80a}) by choosing $\delta$ in such a way that
\[
\tau\left((2 h)^{k+1} \delta\right)=\frac{\log (2 / \alpha)}{n c_{k+1}} .
\]\end{proof}

\begin{proof}[\bf Proof of Theorem \ref{te413}]  As in the proof of Theorems \ref{pra3} and \ref{pra4}, it will suffice to estimate $\mathbb{E} T_{k+1, n}^{2}$. To this end, we have from (\ref{53a})
	\[
	\mathbb{E} W_{k+1, n}^{2}(\theta)=n \mathbb{E} Z_{k+1}^{2}(\theta), \quad-1 \leq \theta \leq 1 .
	\]
	We thus have from (\ref{54a}), Schwarz's inequality and Fubini's theorem
	\begin{align}
		&\mathbb{E}T_{k+1, n}^{2}=\frac{1}{n^{2}(2 h)^{2(k+1)}} \mathbb{E}\left(\mathbb{\widetilde{E}} W_{k+1, n}\left(\widetilde{S}_{1}\right)\right)^{2} \leq \frac{1}{n^{2}(2 h)^{2(k+1)}} \mathbb{E}\times\mathbb{\widetilde{E}} W_{k+1, n}^{2}\left(\widetilde{S}_{1}\right) \nonumber\\
		&=\frac{1}{n(2 h)^{2(k+1)}} \mathbb{E}\times\widetilde{\mathbb{E}} Z_{k+1}^{2}\left(\widetilde{S}_{1}\right) \leq\binom{ 2 k}{k} \frac{1}{n(2 h)^{2 k+1}} D_{k + 2, k, h}(\rho ; x), \label{71a}
	\end{align}
	where the last inequality follows from (\ref{57a}). This completes the proof. \end{proof}

\section*{Acknowledgments}
First and third authors were supported by Gobierno de Aragón Research Project E48\_23R, and second author by the Research Project S41\_23R. First and third authors are also supported by Spanish Research Project PID2024-155364NB-I00 (MINECO/FEDER).


\begin{thebibliography}{4}
					\bibitem{adalra}\textsc{Adell, J.A., Alcalá, J.T., and Sangüesa, C.} (2025). Random positive linear operators and their applications to nonparametric statistics. \textit{TEST} \textbf{34} 981--1011.	
	\bibitem{adsaup}\textsc{Adell, J.A. and Sangüesa, C.} (2001). Upper estimates in direct inequalities for Bernstein-type operators. \textit{J. Approx. Theory} \textbf{109} 229--241.
	
				\bibitem{alto}
	\textsc{Altomare, F. and Campiti, M.} (1994). \textit{Korovkin-type Approximation Theory and Its Applications}, De Gruyter, Berlin.
	

\bibitem{azanot}\textsc{Azzalini, A.} (1981). A note on the estimation of a distribution function and quantiles by a kernel method. \textit{Biometrika}, \textbf{68} 326--328.
	\bibitem{baappl}\textsc{Babu, G.J., Canty, A.J., and Chaubey, Y.P.} (2002). Application of Bernstein polynomials for smooth estimation of a distribution and density function. \textit{J. Stat. Plan. Inf.} \textbf{105} 377--392.
				
\bibitem{calo}
\textsc{Calonico, S., Cattaneo, M. D., and Farrell, M. H.} (2018).
On the effect of bias estimation on coverage accuracy in nonparametric inference. \textit{J. Am. Stat. Assoc.}
\textbf{113} 767--779

\bibitem{clhake}
\textsc{Cline, D. B. H. and Hart, J.D.} (1991).
Kernel estimation of densities with discontinuities or discontinuous derivatives. \textit{Statist.}
\textbf{22} 69--84

\bibitem{dvkias}
		\textsc{Dvoretzky, A, Kiefer, J.C. and Wolfowitz, J.} (1956).
		Asymptotic minimax character of the sample distribution function and of the classical multinomial estimator. \textit{Ann. Math. Statist.}
		\textbf{33} 642--669
		
				\bibitem{duwean}
\textsc{Dümbgen, L. and Wellner, J.A.} (2023).
A new approach to test and confidence bands for distribution functions. \textit{Ann. Statist.}
\textbf{51} 260--289		

				\bibitem{feanin}
	\textsc{Feller, W.} (1997). \textit{An Introduction to Probability Theory and Its Applications}, 2nd. Ed. Vol. II
	Wiley, New York.
	
		\bibitem{fuhike}\textsc{Funke, B. and Hirukawa, M.} (2019).
	Nonparametric estimation and testing on discontinuity of positive supported densities: a kernel truncation approach. \textit{Econom. Stat.}
	\textbf{9} 156-170.
	
	
	\bibitem{ginico}\textsc{Giné, E. and Nickl, R.} (2010).
Confidence bands in density estimation.  \textit{Ann. Statist.}
\textbf{38} 1122--1170.

	\bibitem{hall}\textsc{Hall, P.} (1992).
Effect of bias estimation on coverage accuracy of bootstrap confidence intervals for a probability density.  \textit{Ann. Statist.}
\textbf{20} 675--694.

\bibitem{jothep}	\textsc{Jones, M.C.} (1990).
		The performance of kernel density functions in kernel distribution function estimation. \textit{Statist. Probab. Lett.}
		\textbf{9} 129--132.
				\bibitem{kwexte}
						\textsc{Kwon, E.G.} (1995).
				Extension of H\"{o}lder's inequality (I). \textit{Bull. Aust. Math. Soc.}
				\textbf{51} 369--375.
					\bibitem{leones}
		\textsc{Leblanc, A.} (2012).
		On estimating distribution functions using Bernstein polynomials. \textit{Ann. Inst. Stat. Math.}
		\textbf{64} 919--943.
							\bibitem{maloca}
		\textsc{Maillard, O.-A.} (2021).
		Local Dvoretzky-Kiefer-Wolfowitz confidence bands.  \textit{Math. Methods Statist.}
		\textbf{30} 16--46.
				\bibitem{mathet}
		\textsc{Massart, P.} (1990).
		The tight constant in the Dvoretzky-Kiefer-Wolfowitz inequality. \textit{Ann. Probab.}
		\textbf{18} 1269--1283.
		
					\bibitem{paones}
	\textsc{Parzen, E.} (1962).
On estimation of a probability density function and mode. \textit{Ann. Statist.}
	\textbf{33} 1065--1076.
	
\bibitem{parolo}		\textsc{Patschkowski, T. and Rohde, A.} (2019).
Locally adaptive confidence bands. \textit{Ann. Statist.}
\textbf{47} 349--381.

	\bibitem{rorema}
	\textsc{Rosenblatt, M.} (1956).
	Remarks on some nonparametric estimates of a density function. \textit{Ann. Math. Statist.}
	\textbf{27} 832--837.
	
	\bibitem{scesti}		\textsc{Schuster, E. F.} (1969).
Estimation of a probability density function and its derivatives. \textit{Ann. Math. Statist.}
	\textbf{40} 1187--1195.		

    \bibitem{scott}	\textsc{Scott, D. W.} (2015).  \textit{Multivariate Density Estimation: Theory, Practice, and Visualization}.
		John Wiley \& Sons, New York.
		
		\bibitem{shweem}	\textsc{Shorack, G.R. and Wellner, J.A.} (2009). \textit{Empirical Processes with Applications to Statistics}.  Classics in Applied Mathematics, Philadelphia, PA. Reprint of the 1986 original [MR0838963].
		
			\bibitem{siweak}		\textsc{Silverman, B.W.} (1978).
		Weak and strong uniform consistency of the kernel estimate of a density and its derivatives. \textit{Ann. Statist.}
		\textbf{6} 177--184.	
		
	\bibitem{sidens}	\textsc{Silverman, B.W.} (1986).  \textit{Density Estimation for Statistics and Data Analysis}.
		Chapman \& Hall, London.
		
	\bibitem{stpago}
	\textsc{Stepanova, N.A. and Pavlenko, T.} (2018).
	 Goodness-of-fit tests based on sup-functionals of weighted empirical processes. \textit{Teor. Veroyatn. Primen.}
	\textbf{63} 358--388.
		
			\bibitem{tekern}		\textsc{Tenkorang, N., Ohene-Obeng, K.A., and Su, S.} (2025).
	 Kernel density estimation and convolution revisited. \textit{ArXiv: 2510.19960}

	\bibitem{tsyba}	\textsc{Tsybakov, A. B.} (2009).  \textit{Introduction to Nonpapametric Estimation}.
	 Springer, New York.

            \bibitem{wajoke}   \textsc{Wand, M.P. and Jones, M.C.} (1995).  \textit{Kernel Smoothing}.
	 Chapman \& Hall, London.
	
\bibitem{walt}
\textsc{Walther, G., Ali, A., Shen, X., and Boyd, S.} (2022).
Confidence bands for a log-concave density.  \textit{J. Comput. Graph. Stat.}
\textbf{31} 1426--1438.	


\bibitem{wachsm}
\textsc{Wang, J., Cheng, F., and Yang, L.} (2013).
Smooth simultaneous confidence bands for cumulative distribution functions. \textit{J. Nonparametr. Stat.}
\textbf{25} 395--407.	
	
	

				\bibitem{zhesti}	\textsc{Zhang, S., Li, Z., and Zhang, Z.} (2020).
	Estimating a distribution function at the boundary. \textit{Austrian J. Stat.}
		\textbf{49} 1--23.




		
			\end{thebibliography}
\end{document}